\documentclass[a4paper,10pt]{amsart}

\usepackage{stmaryrd} 

\usepackage[colorlinks=true, linkcolor={black},citecolor={black}]{hyperref}

\usepackage{verbatim,exscale,amssymb,amsthm,mathrsfs,amsrefs}

\numberwithin{equation}{section}

\usepackage[all]{xy}

\xyoption{ps}
\SelectTips{cm}{}

\newdir{ >}{{}*!/-5pt/\dir{>}}

\newcommand{\eg}{\emph{e.g.}}
\newcommand{\ie}{\emph{i.e.}}
\newcommand{\cf}{\emph{cf.}}

\newcommand{\at}{\!\!\downharpoonright}
\newcommand{\variable}{\,\textrm{\textvisiblespace}\,}

\newcommand{\id}{\mathrm{id}}

\newcommand{\loc}{\mathrm{loc}}

\newcommand{\Res}[2]{\mathop{\mathrm{Res}^{#1}_{#2}}}

\newcommand{\Con}{\mathrm{Con}}
\newcommand{\conj}{\mathrm{conj}}
\newcommand{\con}{\mathrm{con}}
\newcommand{\Ind}[2]{\mathop{\mathrm{Ind}_{#1}^{#2}}}

\newcommand{\res}[2]{\mathop{\mathrm{res}^{#1}_{#2}}}
\newcommand{\ind}[2]{\mathop{\mathrm{ind}_{#1}^{#2}}}

\newcommand{\inthom}{\mathrm{\underline{Hom}}}
\newcommand{\Mor}{\mathrm{Mor}}

\newcommand{\N}{\mathbb{N}}
\newcommand{\op}{\mathrm{op}}
\newcommand{\pt}{\mathrm{pt}}

\newcommand{\can}{\mathrm{can}}

\newcommand{\unit}{\mathbf{1}} 

\newcommand{\Cstar}{\mathrm{C}^*\!}

\newcommand{\Cstaralg}{\mathsf{C}^*\!\mathsf{alg}}
\newcommand{\Cstarsep}{\mathsf{C}^*\!\mathsf{sep}}

\newcommand{\Ext}{\mathrm{Ext}}
\newcommand{\Tor}{\mathrm{Tor}}

\newcommand{\KK}{\mathsf{KK}}

\newcommand{\Cell}{\mathsf{Cell}}

\newcommand{\Gr}{\mathsf{Gr}}

\renewcommand{\mod}{\mathsf{mod}}

\newcommand{\Mod}{\mathsf{Mod}}
\newcommand{\GrMod}{\mathsf{GrMod}}

\newcommand{\Mack}{\mathsf{Mac}}
\newcommand{\Mac}{\mathsf{Mac}}
\newcommand{\boxpr}{{\boxempty}}

\newcommand{\Projobj}{\mathsf{Proj}}

\newcommand{\Ab}{\mathsf{Ab}}

\newcommand{\set}{\mathsf{set}}

\newcommand{\Z}{\mathbb{Z}}
\newcommand{\R}{\mathbb{R}}
\newcommand{\C}{\mathbb{C}}
\newcommand{\Q}{\mathbb{Q}}

\theoremstyle{definition}

\newtheorem{defi}[equation]{Definition}

\newtheorem*{conv*}{Conventions}
\newtheorem*{ack}{Acknowledgements}
\newtheorem*{Related work}{Related work}

\theoremstyle{theorem}

\newtheorem{thm}[equation]{Theorem}
\newtheorem{lemma}[equation]{Lemma}
\newtheorem{thm-defi}[equation]{Theorem-Definition}
\newtheorem{prop}[equation]{Proposition}
\newtheorem{cor}[equation]{Corollary}

\newtheorem*{thm*}{Theorem}
\newtheorem*{lemma*}{Lemma}
\newtheorem*{cor*}{Corollary}
\newtheorem*{conj*}{Conjecture}
\newtheorem*{question*}{Question}

\theoremstyle{remark}

\newtheorem{notation}[equation]{Notation}
\newtheorem{remark}[equation]{Remark}
\newtheorem*{remark*}{Remark}
\newtheorem{remarks}[equation]{Remarks}
\newtheorem{example}[equation]{Example}

\begin{document}
\title{Equivariant Kasparov theory of finite groups via Mackey functors}
\author{Ivo Dell'Ambrogio} 
\address{Universit\"at Bielefeld, Fakult\"at f\"ur Mathematik, BIREP Gruppe, Postfach 10\,01\,31, 33501 Bielefeld, Germany}
\email{ambrogio@math.uni-bielefeld.de}
\date{}


\maketitle

\begin{abstract}
Let $G$ be any finite group. In this paper, we systematically exploit general homological methods in order to reduce the computation of $G$-equivariant $KK$-theory to topological equivariant $K$-theory.
The key observation is that the functor on $\KK^G$ that assigns to a $G$-$\Cstar$-algebra $A$ the collection of its $K$-theory groups $\{ K^H_*(A) :  H\leqslant G \}$ admits a lifting to the abelian category of $\Z/2$-graded Mackey modules over the representation Green functor for~$G$; moreover, this lifting is the universal exact homological functor for the resulting relative homological algebra in~$\KK^G$. 
It follows that there is a spectral sequence abutting to $\KK^G_*(A,B)$, whose second page displays Ext groups computed in the category of Mackey modules.
Thanks to the nice properties of Mackey functors, we obtain a similar K\"unneth spectral sequence which computes the equivariant $K$-theory groups of a tensor product $A\otimes B$. Both spectral sequences behave nicely if~$A$  belongs to the localizing subcategory of $\KK^G$ generated by the algebras $C(G/H)$ for all subgroups $H\leqslant G$.

\end{abstract}

\tableofcontents

\section{Introduction}
\label{intro}

The theory of Mackey functors (\cite{webb:guide, dress:contributions ,bouc:green_functors, lewis:mimeo,thevenaz-webb:structure}\ldots) has proved itself to be a powerful conceptual \emph{and} computational tool in many branches of mathematics: 
group cohomology, equivariant stable homotopy, algebraic K-theory of group rings, algebraic number theory, etc.; in short, any theory where one has finite group actions and induction/transfer maps. We refer to the survey article~\cite{webb:guide}.

Equivariant Kasparov theory $\KK^G$ (\cite{kasparov_first,phillips:free, meyer_nest_bc, meyer_cat}\ldots), although typically more preoccupied with topological groups or infinite discrete groups, is already quite interesting when~$G$ is a finite group -- see for instance the work of C.\,H.\ Phillips~\cite{phillips:free} on the freeness of $G$-actions on $\Cstar$-algebras. There exist induction maps for $\KK^G$-theory, so it is natural to ask whether the theory of Mackey functors has anything useful to say in this context.
As we will shortly see, the answer is definitely~``Yes''.

Recall that for a finite (or, more generally, compact) group~$G$ and every $G$-$\Cstar$-algebra~$A$, we have the natural identification $\smash{\KK^G_*(\C,A)\cong K^G_*(A)}$ with $G$-equivariant topological K-theory~$K^G_*$, which generalizes Atiyah and Segal's classical $G$-equivariant vector bundle K-cohomology of spaces and has  similar properties. Consequently, equivariant K-theory is often easier to compute than general equivariant KK-theory.
In view of all this, it is natural to ask:

\begin{question*}
To what extent, and how, is it possible to reduce the computation of equivariant KK-theory groups to that of equivariant K-theory groups?
\end{question*}

In order to answer this question precisely, the following two observations will be crucial. First, for any fixed $G$-$\Cstar$-algebra~$A$ the collection of all its equivariant topological K-theory groups $K^H_*(\Res{G}{H} A)\cong \KK^G_*(C(G/H),A)$ ($H\leqslant G$) and of the associated restriction, induction and conjugation maps, forms a graded Mackey functor for~$G$, that we denote~$k^G_*(A)$. In fact, $k^G_*(A)$ carries the structure of a Mackey module over the representation Green functor~$R^G$, and if we denote by $R^G\textrm-\Mack_{\Z/2}$ the category of $\Z/2$-graded modules over~$R^G$ -- which is a perfectly nice  Grothendieck tensor abelian category -- we obtain in this way a lifting of K-theory to a homological functor $k^G_*\colon \KK^G\to R^G\textrm-\Mack_{\Z/2}$ on the triangulated Kasparov category of separable $G$-$\Cstar$-algebras. 
The second basic observation is that~$k^G_*$ is the ``best'' such lifting of K-theory to some abelian category approximating $\KK^G$:
in the technical jargon of \cite{meyernest_hom, meyer:hom}, the functor $k^G_*$ (or more precisely: its restriction to countable modules) is the universal stable homological functor on~$\KK^G$ for the relative homological algebra defined by the K-theory functors $\{K^H_*\circ \Res{G}{H}\}_{H\leqslant G}$. 
Another, but equivalent, way to formulate this second observation is the following: the category $R^G\textrm-\Mack$ of $R^G$-Mackey modules is equivalent to the category of additive contravariant functors $\mathsf{perm}^G\to \Ab$, where $\mathsf{perm}^G$ denotes the full subcategory in~$\KK^G$ of \emph{permutation algebras}, \emph{i.e.}, those of the form~$C(X)$ for~$X$ a finite $G$-set. (See \S\ref{subsec:BBKK}.)

Once all of this is proved, it is a straightforward matter to apply the general techniques of relative homological algebra in triangulated categories (\cite{christensen:ideals,beligiannis:relative, meyer:hom , meyer_homom}) in order to obtain a universal coefficient spectral sequence which will provide our answer to the above question. 
By further exploiting the nicely-behaved tensor product of Mackey modules, we similarly obtain a K\"unneth spectral sequence for equivariant K-theory, which moreover has better convergence properties.

The natural domain of convergence of these spectral sequences consists of \emph{$G$-cell algebras}, namely, those algebras contained in the localizing triangulated subcategory of $\KK^G$ generated by $\mathsf{perm}^G$ (equivalently: generated by the algebras~$C(G/H)$ for all subgroups $H\leqslant G$). We note that the category, $\smash{\Cell^G}$, of $G$-cell algebras is rather large. For instance it contains all abelian separable $G$-$\Cstar$-algebras and is closed under all the classical ``bootstrap'' operations (see Remark~\ref{remark:closure}).

We can now formulate our results precisely.

\begin{thm*}[Thm.~\ref{thm:UCT}]
Let $G$ be a finite group. 
For every~$A$ and~$B$ in~$\KK^G$, and depending functorially on them, there exists a cohomologically indexed right half-plane spectral sequence of the form
\[
E^{p,q}_2 = \Ext^p_{R^G}( k^G_*A , k^G_*B )_{-q} 
\quad
\stackrel{n= p + q}{\Longrightarrow} 
\quad
\KK^G_n(A, B) \,.
\]
The spectral sequence converges conditionally whenever $A$ is a $G$-cell algebra, and converges strongly if moreover~$A$ is such that $\KK^G(A,f)=0$ for every morphism~$f$ which can be written, for each $n\geqslant1$, as a composition of~$n$ maps each of which vanishes under~$k^G_*$. 
If $A$ is a $G$-cell algebra and the $R^G$-module $k^G_*A$ has a projective resolution of finite length~$m\geqslant1$, then the spectral sequence is confined in the region $0\leqslant p \leqslant m+1$ and thus collapses at the page~$E^{*,*}_{m+1}=E^{*,*}_\infty$.
\end{thm*}

The groups displayed in the second page~$E^{*,*}_2$ are the homogeneous components of the graded Ext functors of $R^G\textrm-\Mack_{\Z/2}$. Concretely, for $M,N\in R^G\textrm-\Mack_{\Z/2}$
\[
\Ext^n_{R^G}(M,N)_\ell
=
\bigoplus_{i+j=\ell} \Ext^n_{R^G}(M_i, M_j)
\qquad (\ell\in \Z/2  , n\geqslant 0)
\]
where the right-hand-side $\Ext^n_{R^G}(\variable,M_j)$ denotes the $n$-th right derived functor of the Hom functor ${R^G}\textrm-\Mack (\variable,M_j)$ on the abelian category of $R^G$-Mackey modules. The latter category, as well as their objects $k^G_*(A)$, are explained in great detail in Sections~\ref{sec:mackey_green} and~\ref{sec:K_mack}.

\begin{thm*}[Thm.~\ref{thm:KT}]
Let $G$ be a finite group.
For all separable $G$-$\Cstar$-algebras~$A$ and~$B$, and depending functorially in them, there is a homologically indexed right half-plane spectral sequence of the form
\[
E^2_{p,q} = \Tor_{p}^{R^G}(k^G_*A,k^G_*B)_{q}
\quad \stackrel{n=p+q}{\Longrightarrow} \quad K^G_n(A\otimes B)
\]
which strongly converges if $A$ is a $G$-cell algebra.  
If moreover $A$ is such that $k^G_*A$ has a projective resolution of finite length~$m\geqslant1$, then the spectral sequence is confined in the region $0\leqslant p\leqslant m$ and thus collapses at the page~$E_{*,*}^{m+1}=E_{*,*}^{\infty}$.

\end{thm*}

Now the second page~$E^2_{*,*}$ contains the left derived functors of the tensor product~$\boxpr_{R^G}$ of $\Z/2$-graded $\smash{R^G}$-Mackey modules, which is explained in~\S\ref{subsec:tensor}.

From these spectral sequences there follow the usual consequences and special cases. 
Here we only furnish, as a simple illustration, the following vanishing result.

\begin{thm}
\label{thm:appl}
Let $A$ and~$B$ be two $G$-$\Cstar$-algebras for a finite group~$G$, and assume that either~$A$ or~$B$ is a $G$-cell algebra.
If $K^E_*(\Res{G}{E} A)=0$ for all elementary subgroups $E$ of~$ G$, then $K^G_*(A\otimes B)=0$. 
\end{thm}

\begin{proof}
The hypothesis on K-theory implies that $k^G_*(A)=0$ (see Lemma \ref{lemma:brauer_artin} below) and therefore the second page of the K\"unneth spectral sequence is zero.
By symmetry of the tensor product we may assume that $A\in\Cell^G$, and we conclude that $K^G_*(A\otimes B)=0$ by the strong convergence.
\end{proof}

\begin{Related work}
To our knowledge,~\cite{lewis-mandell:uct} is the only published work where spectral sequences are systematically computed in abelian categories of Mackey modules; this is done in the context of equivariant stable homotopy. There may be some overlap between their results and ours; specifically, it should be possible to use \emph{loc.\,cit.}\  to reprove our results in the special case of \emph{commutative} $\Cstar$-algebras.
We also mention that~\cite{ventura} performs explicit computations of Ext functors in the category of Mackey modules over~$R^G$ for some small groups~$G$.  

For $G$ a connected Lie group with torsion-free fundamental group, and for sufficiently nice $G$-$\Cstar$-algebra, there are the K\"unneth and universal coefficient spectral sequences of~\cite{rosenberg-schochet:equiv}, which are computed in the ordinary module category over the complex representation ring of~$G$. 
It seems plausible that a unified treatment of their and our results might be both obtainable and desirable, possibly in terms of Mackey functors for compact Lie groups (\emph{cf}.\ Remark~\ref{rem:stable_htpy}).

Quite recently, a universal coefficient short exact sequence was constructed in~\cite{koehler:uct} for $\KK^{G}$ when~$G$ is a cyclic group of prime order. The invariant used in \emph{loc.\,cit.}\ is a slightly more complicated lifting of K-theory than our Mackey module~$k^G_*$, and contains more information. The range of applicability is the same though: the first algebra must belong to $\Cell^G$.
\end{Related work}

\begin{conv*}
For simplicity, we will work only with complex $\Cstar$-algebras and complex group representations, although the alert reader will see without any trouble how to adapt all results to the real case. 
Our notation $\Res{G}{H}$ for the restriction functor from $G$ to~$H$ is at odds with \eg~\cite{meyer_nest_bc}, where $\Res{H}{G}$ is used instead, but is compatible with the common indexing conventions in the context of Mackey functors.
We always write $\mathcal C(X,Y)$ for the set of morphisms from the object $X$ to the object~$Y$ in a category~$\mathcal C$.
We use the short-hand notations ${}^gH:=gHg^{-1}$ and $H^g:=g^{-1}Hg$ for the conjugates of a subgroup $H\leqslant G$.
If $H,L\leqslant G$ are subgroups, the notation $[H\backslash G/L]$ denotes a full set of representatives of the double cosets $HgL\subseteq G$.

\end{conv*}

\begin{ack}
Our warm thanks go to Serge Bouc for several illuminating discussions on the virtues and vices of Mackey functors. 
\end{ack}

\section{$G$-cell algebras}

After some recollections on the equivariant Kasparov category, we introduce the subcategory of $G$-cell algebras and derive its first properties.

\subsection{Restriction, induction and conjugation}
\label{subsec:res_ind_con}

Let $\KK^G$ be the Kasparov category of separable $G$-$\Cstar$-algebras, for a second countable locally compact group~$G$.
 We refer to the articles \cite{meyer_cat , meyer_nest_bc} for an account of KK-theory considered from the categorical point of view; therein the reader will find proofs or references for the facts recalled in this subsection.
For each~$G$, the category $\KK^G$ is additive and has arbitrary countable coproducts, given by the $\Cstar$-algebraic direct sums $\bigoplus_iA_i$ on which~$G$ acts coordinatewise.
Moreover, it is equipped with the structure of a triangulated category (see~\cite{meyer_nest_bc}, esp.\ Appendix~A); in particular every morphism $f\in \KK^G(A,B)$ fits into a \emph{distinguished triangle} $A\to B\to C\to A[1]$, and the collection of distinguished triangles satisfies a set of axioms that capture the homological behaviour of KK-theory.
Here the \emph{shift} (or suspension, translation) functor $A\mapsto A[1]$ is the endoequivalence of $\KK^G$  given by $A[1]=C_0(\R)\otimes A$. By the Bott isomorphism $C_0(\R)\otimes C_0(\R)\cong C_0(\R)$, this functor is its own quasi-inverse.
Using a standard trick, it is always possible to ``correct'' the shift functor making it a (strict) automorphism (see~\cite{meyer_nest_bc}*{\S2.1} and~\cite{keller-vossieck:sous}*{\S2}). Therefore, in order to simplify notation, we shall pretend that~$(\variable)[1]\colon \KK^G\to \KK^G$ is strictly invertible, with $[2] \smash{\stackrel{\textrm{def.}}{=}} [1]\circ [1] =\id_{\KK^G} $.

The triangulated category $\KK^G$ is also endowed with a compatible symmetric monoidal structure $\KK^G\times \KK^G\to \KK^G$, which is induced by the spatial tensor product $A\otimes B$ of $\Cstar$-algebras on which~$G$ acts diagonally (in fact, we have already used this to define the shift functor). The unit object $\unit^G$ (or simply~$\unit$ if no confusion arises) is the algebra~$\C$ of complex numbers with the trivial $G$-action.

The tensor product is not the only construction at the $\Cstar$-algebraic level that extends to a triangulated functor on the Kasparov categories. 
For instance, there is an evident restriction functor 
\[
\Res{G}{H} \colon \KK^G \to \KK^H
\]  
for every subgroup $H\leqslant G$, which commutes with coproducts and is also (strict) symmetric monoidal: $\Res{G}{H}(A\otimes B)=\Res{G}{H}(A)\otimes \Res{G}{H}(B)$ and $\Res{G}{H}(\unit^G)=\unit^H$. 
If~$H$ is closed in~$G$ there is also a coproduct-preserving induction functor 
\[
\Ind{H}{G} \colon \KK^H \to \KK^G
\]  
which on each $H$-$\Cstar$-algebra $A\in \KK^H$ is given by the function $\Cstar$-algebra
\[
\Ind{H}{G}(A) = \{  G\stackrel{\varphi}{\to} A \mid h\varphi(xh)=\varphi(x) \,\forall x\in G,h\in H ; (xH \mapsto \|\varphi(x)\|)\in C_0(G/H) \}
\]
equipped with the $G$-action $(g\cdot \varphi)(x):= \varphi(g^{-1}x)$ ($g,x\in G$).
If $G/H$ is discrete, then induction is left adjoint to restriction, \ie, there is a natural isomorphism
\[
\KK^G( \Ind{H}{G}A, B)\cong \KK^H(A, \Res{G}{H}B)
\]
for all $A\in \KK^G$ and $B\in \KK^H$. 
Interestingly, if instead $G/H$ is compact then induction is \emph{right} adjoint to restriction. 
There is also a \emph{Frobenius isomorphism}
\begin{align}
\label{frob:Cstar}
 \Ind{H}{G}(A) \otimes B
 \cong
 \Ind{H}{G}(A \otimes \Res{G}{H}(B)) 
\end{align}
natural in $A\in \KK^H$ and $B\in \KK^G$.

The induction and restriction functors will be used constantly in this article. 
For every subgoup $H\leqslant G$ and every element $g\in G$, we will also consider the \emph{conjugation functor}
\[
{}^g(\variable) \colon \KK^H \to \KK^{\,{}^g\!H}
\]
which sends the $H$-$\Cstar$-algebra~$A$ to the ${}^g\! H$-$\Cstar$-algebra ${}^gA$ whose underlying $\Cstar$-algebra is just~$A$, equipped with the ${}^gH$-action  $ghg^{-1}a:= ha$ ($h\in H, a\in A$). 
Like restriction -- and for the same reasons -- each conjugation functor preserves coproducts, triangles and tensor products. Moreover, it is an isomorphism of tensor triangulated categories with inverse ${}^{g^{-1}}\!(\variable)$.

\subsection{The category of $G$-cell algebras}
For every closed subgroup $H\leqslant G$, we have the ``standard orbit'' $G$-$\Cstar$-algebra $C_0(G/H)$.
The idea is that $G$-cell algebras are those (separable) $G$-$\Cstar$-algebras that can be produced out of these by applying all the standard operations of triangulated categories.

Although we will be mostly concerned with finite groups, in this subsection we briefly study $G$-cell algebras  in greater generality, for future reference.

\begin{defi} \label{defi:CWG}
We define the \emph{Kasparov category of $G$-cell algebras} to be the localizing triangulated subcategory of $\KK^G$ generated by all~$C_0(G/H)$, in symbols: 
\begin{equation*}
\Cell^G:= \big\langle \; C_0(G/H) \mid H\leqslant G \;  \big\rangle_{\loc} \quad \subseteq \quad \KK^G.
\end{equation*}
This means that $\Cell^G$ is the smallest triangulated subcategory of~$\KK^G$ that contains all~$C_0(G/H)$ and is closed under the formation of infinite direct sums.
\end{defi}

\begin{remark}
The same notion of $G$-cell algebra is considered in~\cite{koehler:uct}, and is proposed as a $\KK$-analogue of $G$-CW-complexes. 
An even better analogy would be ``cellular objects'' in a (model) category of equivariant spaces, where the order of attachment of the cells is completely free, like here.
In order to obtain a more rigid notion of noncommutative  $G$-CW-complexes -- which would serve similar purposes as in the commutative case -- one should rather extend to the equivariant setting the definition of noncommutative CW-complexes of~\cite{nccw}. 
\end{remark}

\begin{remark}
\label{remark:closure}
The class $\Cell^G$ contains many interesting $G$-$\Cstar$-algebras, although this may not be apparent from the definition. 
For instance if~$G$ is a compact (non necessarily connected) Lie group, by \cite{koehler:uct}*{Thm.~9.5} the class $\Cell^G$ contains all separable commutative $G$-$\Cstar$-algebras and is closed under the usual bootstrap operations, in the sense that it enjoys the following closure properties:
\begin{enumerate}
\item For every extension $J\rightarrowtail A \twoheadrightarrow B$ of nuclear separable $G$-$\Cstar$-algebras, if two out of $\{J,A,B\}$ are in $\Cell^G$ then so is the third.
\item $\Cell^G$ is closed under the formation (in the category of $G$-$\Cstar$-algebras and $G$-equivariant $*$-homomorphisms) of colimits of countable inductive system of nuclear separable $G$-$\Cstar$-algebras.  
\item $\Cell^G$ is closed under exterior equivalence of $G$-actions.
\item $\Cell^G$ is closed under $G$-stable isomorphisms.
\item $\Cell^G$ is closed under the formation of crossed products with respect to~$\Z$- and~$\R$-actions that commute with the given $G$-action.
\end{enumerate}
Next, we show that much of the functoriality of equivariant KK-theory descends to $G$-cell algebras. 
\end{remark}

\begin{lemma}
\label{label:closure_loc_tens}
Let $\mathcal T$ be a triangulated category equipped with a symmetric tensor product which preserves coproducts \textup(whatever are available in~$\mathcal T$\textup) and triangles.  
Then $\langle \mathcal E \rangle_\loc \otimes \langle\mathcal F \rangle_\loc \subseteq \langle \mathcal E\otimes \mathcal F\rangle_\loc$ for any two subclasses $\mathcal E,\mathcal F\subseteq \mathcal T$.
\end{lemma}

\begin{proof}
First, we claim that 
\begin{equation}
\label{partial_cont}
\langle\mathcal E\rangle_\loc \otimes \mathcal F \subseteq \langle \mathcal E \otimes \mathcal F \rangle_\loc.
\end{equation}
For every object $B\in \mathcal T$, the functor $\variable \otimes B$ commutes with coproducts and triangles by hypothesis. 
Thus
$\mathcal S_B:=\{A \in \mathcal T\mid A\otimes B \subseteq \langle\mathcal E\otimes B \rangle_\loc \}$ 
is a localizing triangulated subcategory of~$\mathcal T$, which moreover contains~$\mathcal E$; hence $\langle \mathcal E\rangle_\loc \subseteq \mathcal S_B$.
Therefore for every $B\in \mathcal F$ we have
$\langle \mathcal E\rangle_\loc \otimes B 
 \subseteq \langle \mathcal E \otimes B \rangle_\loc
  \subseteq \langle \mathcal E \otimes \mathcal F\rangle_\loc $, from which~\eqref{partial_cont} follows.
Similarly, for every $A\in \mathcal T$ we see that $\mathcal U_A:=\{B\in \mathcal T\mid A \otimes B \subseteq \langle \mathcal E\otimes \mathcal F\rangle_\loc \}$ is localizing, and therefore also 
 $\mathcal U:=\{B\in \mathcal T \mid \langle \mathcal E\rangle_\loc \otimes B \subseteq \langle \mathcal E\otimes \mathcal F\rangle_\loc \} = \bigcap_{A\in \langle \mathcal E\rangle_\loc} \mathcal U_A$. 
 By~\eqref{partial_cont}, $\mathcal U$ contains $\mathcal F$, so it must contain $\langle \mathcal F\rangle_\loc$. This was precisely the claim.
\end{proof}

\begin{prop}
\label{prop:cell_G}
Assume that~$G$ is a discrete group or a compact Lie group. 
Then~$\Cell^G$ is a tensor-triangulated subcategory of~$\KK^G$. 
Moreover, all restriction, induction and conjugation functors  $\Res{G}{H}$,  $\Ind{H}{G}$ and~${}^g(-)$, descend to the appropriate Kasparov subcategories of cell algebras. 
\end{prop}

\begin{proof}
The tensor unit $\unit^G=C_0(G/G)$ belongs to $\Cell^G$.
Moreover, for all closed subgroups $H,L\leqslant G$, the algebra 
$C_0(G/H)\otimes C_0(G/L) \cong C_0(G/H \times G/L)$ belongs again to $\Cell^G$. Indeed, if $G$ is a compact Lie group this follows from Remark~\ref{remark:closure} because the algebra is commutative; if~$G$ is discrete, then it follows simply by applying  the coproducts-preserving functor~$C_0$ to the orbit decomposition of $G$-sets
\[
G/H \times G/L 
\cong
 \coprod_{x\in [H\backslash G/L]} G/(H\cap {}^xL) \,.
\] 
We conclude by Lemma~\ref{label:closure_loc_tens}, with $\mathcal E=\mathcal F:= \{C_0(G/H) \mid H\leqslant G \textrm{ closed}\}$, that 
$\Cell^G \otimes \Cell^G
 = \langle \mathcal E \rangle_\loc \otimes \langle \mathcal E \rangle_\loc
  \subseteq \langle \mathcal E\otimes \mathcal E \rangle_\loc
  \subseteq \langle \mathcal E \rangle_\loc
  = \Cell^G$.
This proves that $\Cell^G$ is a tensor subcategory of $\KK^G$.  
  
 The induction functors satisfy $\Ind{H}{G}\circ \Ind{L}{H}\cong \Ind{L}{G}$ for all all  $L\leqslant H\leqslant G$, and each $\Ind{G}{H}$ commutes with triangles and coproducts. Thus $\Ind{H}{G}(C_0(H/L))=\Ind{H}{G}\circ \Ind{L}{H}(\C) = C_0(G/L)$, and we conclude that $\Ind{H}{G}(\Cell^H)\subseteq \Cell^G$.
Similarly, the identifications ${}^g C_0(G/H)\cong C_0({}^gG/ {}^gH)$ 
in $\KK^{{}^gG}$ for all $H\leqslant G$ show that ${}^g(\Cell^G)\subseteq \Cell^{{}^gG}$
(and thus $\Cell^G\cong \Cell^{{}^gG}$).
Finally, for~$G$ discrete the isomorphism of $H$-$\Cstar$-algebras 
\begin{equation*} \label{eq:mackey_C}
\Res{G}{H}C_0(G/L) 
= C_0(\Res{G}{H} G/L)
\; \;  \cong \bigoplus_{x\in [H\backslash G/L]} C_0(H/H\cap {}^xL)
\end{equation*}
shows that $\Res{G}{H} C(G/L)\in \Cell^H$ for all $L\leqslant G$ and therefore $\Res{G}{H}(\Cell^G)\subseteq \Cell^H$.
If $G$ is a compact Lie group, we notice that $\Res{G}{H}C(G/L)$ is commutative and appeal again to Remark~\ref{remark:closure}.  
\end{proof}


\begin{lemma} \label{lemma:self_dual}
For finite groups $H\leqslant G$, there is an isomorphism
\begin{equation*}
\KK^G(A \otimes C(G/H), B) \cong \KK^G (A, C(G/H)\otimes B)
\end{equation*}
natural in $A,B\in \KK^G$. 
\end{lemma}

\begin{proof}
Since $G/H$ is finite, $\Ind{H}{G}$ and $\Res{G}{H}$ are adjoint to each other on both sides. We obtain the following composition of natural isomorphisms:
\begin{eqnarray*}
\KK^G(A \otimes C(G/H), B) 
 &= & \KK^G \big(A\otimes \Ind{H}{G}(\unit^H), B \big) \\
 &\cong &  \KK^G \big( \Ind{H}{G}(\Res{G}{H}( A) \otimes \unit^H), B \big) \\
 &\cong & \KK^H \big( \Res{G}{H} (A) \otimes \unit^H, \Res{G}{H} (B) \big) \\
 &\cong & \KK^H ( \Res{G}{H} (A) , \unit^H \otimes \Res{G}{H} (B)) \\
 &\cong & \KK^G ( A , \Ind{H}{G}( \unit^H \otimes \Res{G}{H} (B))) \\
 &\cong & \KK^G ( A , \Ind{H}{G}( \unit^H ) \otimes B )) \\
  &= & \KK^G ( A , C(G/H) \otimes B )).
 \end{eqnarray*}
 where we have also used Frobenius~\eqref{frob:Cstar} in the second and sixth lines.
 \end{proof}

The next proposition says that, at least when $G$ is finite, $G$-cell algebras form a rather nice tensor triangulated category.

\begin{prop} \label{prop:cptly_gen}
For every finite group $G$, the tensor triangulated category $\Cell^G$ is generated by the \textup(finite\textup) set $\{C(G/H), C(G/H)[1] \mid H\leqslant G\}$ of  rigid  and compact${}_{\aleph_1}$ objects, in the sense of~\cite{kkGarticle}.
In particular, $ \Cell^G$ is compactly${}_{\aleph_1}$ generated and its subcategory $\Cell^G_c$ of compact${}_{\aleph_1}$ objects coincides with that of its rigid objects, and is therefore a tensor triangulated subcategory. 
\end{prop}

\begin{proof}
To prove the first part, consider the natural isomorphism
\[
\KK^G(C(G/H), A)
\cong \KK^H(\unit, \Res{G}{H}A) = K^H_0(\Res{G}{H}A) \cong K_0(H \ltimes \Res{G}{H}A)
\]
provided by the $\Ind{H}{G} \textrm- \Res{G}{H}$ adjunction and the Green-Julg theorem~\cite{phillips:free}*{\S2.6}. 
If~$A$ is separable, then so is the cross-product $H\ltimes \Res{G}{H}A$, from which it follows that the ordinary K-theory group on the right-hand side is countable; moreover, we see that $\KK^G(C(G/H), \variable)$ sends a coproduct in~$\KK^G$ to a coproduct of abelian groups. These two facts together state precisely that $C(G/H)$ is a compact${}_{\aleph_1}$ object of~$\KK^G$. The same follows immediately for the suspensions~$C(G/H)[1]$.

The second claim follows formally, whenever the set of compact${}_{\aleph_1}$ generators consists of rigid objets and contains the tensor unit. The latter is obvious, since $\unit^G=C(G/G)$, and it follows immediately from Lemma~\ref{lemma:self_dual} that each generator~$C(G/H)$ (and thus also each $C(G/H)[1]$) is rigid -- in fact, self-dual.
\end{proof}

\medbreak
Recall that a finite group is \emph{elementary} if it has the form $P\times C$, where  $C$ is cyclic and~$P$ is a $p$-group for some prime~$p$ not dividing the order of~$C$.

\begin{lemma}
\label{lemma:brauer_artin}
Let $A$ be a $G$-$\Cstar$-algebra for a finite group~$G$. Then:
\begin{enumerate}
\item If $K^E_*(\Res{G}{E} A)=0$ for all elementary subgroups $E\leqslant G$, then $K^G_*(A)=0$.
\item If $K^C_*(\Res{G}{E} A)\otimes_\Z \Q=0$ for all cyclic $C\leqslant G$, then $K^G_*(A)\otimes_\Z \Q=0$.
\end{enumerate}
\end{lemma}

\begin{proof}
(1)
Denote by $\mathcal E(G)$ the set of all elementary subgroups of~$G$.
 Brauer's classical induction theorem (\cite{bensonI}*{Thm.~5.6.4 and p.~188}) says that the homomorphism
 $
\sum_{E} \ind{E}{G} \colon \bigoplus_{E \in \mathcal E(G)} R(E) \to  R(G)
$ 
is surjective, where $R(H)$ denotes the representation ring of a finite group~$H$. In particular, there exist finitely many $E_i\in \mathcal E(G)$ and $x_i \in R(E_i)$, such that $1=\sum_i \ind{E_i}{G}(x_i)$ in~$R(G)$.
Now consider an $A\in \KK^G$ such that $K^H_*(\Res{G}{E}A)=0$ for all $E\in \mathcal E(G)$. 
Since the equivariant K-theory of~$A$ is a Mackey module over the representation ring (see Section~\ref{sec:K_mack}), we compute for every $x\in K^G_*(A)$
\[
x 
= 1_{R(G)}\cdot x 
= \sum_i \ind{E_i}{G}(x_i)\cdot x 
= \sum_i \ind{E_i}{G}( x_i \underbrace{\Res{G}{E_i}(x)}_{=\,0} ) 
=0
\]
by applying the vanishing hypothesis.
The proof of (2) is similar, but now we must use Artin's induction theorem instead.
\end{proof}

For finite~$G$, denote by $\Cell^G_\Q$ the rationalization of the category $\Cell^G$ which is compatible with countable coproducts, \ie, the one obtained by applying \cite{kkGarticle}*{Thm.~2.33} to $\mathcal T:=\Cell^G$ and $S:=(\Z\smallsetminus \{0\})\cdot 1_{\unit}$. 
Thus $\Cell^G_\Q$ is again a compactly${}_{\aleph_1}$ generated tensor triangulated category with the same objects, and it has the property that $\Cell^G_\Q(A,B)\cong \KK^G(A,B)\otimes_\Z\Q$ for all compact${}_{\aleph_1}$ algebras~$A\in \Cell^G_c$ (thus in particular 
$\Cell^G_\Q(C(G/H)[i], B)\cong K^H_i(\Res{G}{H}B)\otimes_\Z\Q$ for every $H\leqslant G$).

By meshing familiar tricks from the theory of Mackey functors and from the theory of triangulated categories, we obtain the following generation result for $G$-cell algebras and rational $G$-cell algebras.

\begin{prop}
\label{prop:elementary}
Let $G$ be a finite group. Then:
\begin{enumerate}
\item $\Cell^G = \langle C(G/H) \mid H \textrm{ is an elementary subgroup of }G  \rangle_\loc $.
\item $\Cell^G_\Q = \langle C(G/H) \mid H \textrm{ is a cyclic subgroup of } G  \rangle_\loc$.
\end{enumerate}
\end{prop}

\begin{proof}

If $\mathcal T$ is a triangulated category with countable coproducts and if $\mathcal E_1,\mathcal E_2 \subseteq \mathcal T_c$ are two countable sets of compact objects which are closed under suspensions and desuspensions, then $\langle \mathcal E_1 \rangle_{\loc}=\langle \mathcal E_2 \rangle_{\loc}$ whenever $\mathcal E_1$ and $\mathcal E_2$ have the same right orthogonal in~$\mathcal T$, \emph{i.e.}, if  
$\mathcal E_1^\perp :=\{ B\in \mathcal T\mid \mathcal T(A,B)=0 \;\forall A\in \mathcal E_1\}$ 
equals $\mathcal E_2^\perp :=\{ B\in \mathcal T\mid \mathcal T(A,B)=0 \;\forall A\in \mathcal E_2\}  $
  (see  \cite{kkGarticle}*{\S2.1} for explanations).
Thus part~(1) follows immediately,  using $\mathcal T=\KK^G$ or $\mathcal T=\Cell^G$, by combining Proposition~\ref{prop:cptly_gen} with Lemma \ref{lemma:brauer_artin}~(1), while part~(2) uses Lemma \ref{lemma:brauer_artin}~(2) instead (and $\mathcal T= \Cell^G_\Q$).
\end{proof}

\section{Recollections on Mackey and Green functors}
\label{sec:mackey_green}

Throughout this section, we fix a finite group~$G$.

Mackey functors, and the related notions of Green functors and modules over them, can be defined from various different point of views. 
The three most important (all of which are treated in detail in~\cite{bouc:green_functors}) are the definition in terms of subgroups of~$G$, that in terms of $G$-sets, and that in terms of functor categories. 

Since we are going to need all three of them, let us proceed without further delay.

\subsection{The subgroup picture}
\label{subsec:mackey}
This is the most concrete of the three points of view.
A \emph{Mackey functor $M$ \textup(for~$G$\textup)} consists of a family of abelian groups~$M[H]$, one for each subgroup~$H\leqslant G$, 
together with a restriction homomophism $\res{H}{L}\colon M[H]\to M[L]$ and an induction homomorphism $\ind{L}{H}\colon M[L]\to M[H]$ for all $L\leqslant H \leqslant G$, and a conjugation homomorphism $\con_{g, H}\colon M[H]\to M[{}^gH]$ for all $g\in G$ and all $H\leqslant G$. These three families of maps must satisfy the following six families of relations:
\begin{align*}
& \res{H}{L} \res{G}{H} = \res{G}{L} 
 \quad, \quad \ind{H}{G} \ind{L}{H} = \ind{L}{G} 
  & 
  (L\leqslant H\leqslant G)  
  \\
& \con_{f,{}^gH} \, \con_{g,H} = \con_{fg, H}  
  & 
  (f,g\in G, \, H\leqslant G) 
  \\
& \con_{g,H} \, \ind{L}{H} = \ind{{{}^gL}}{{{}^gH}} \, \con_{g,L}  
   &
   (g\in G,\, L\leqslant H\leqslant G)
   \\
 & \con_{g,L} \res{H}{L} = \res{{{}^gH}}{{{}^gL}} \con_{g,H}  
  &
   (g\in G,\, L\leqslant H\leqslant G)
    \\
& \ind{H}{H} = \res{H}{H} = \con_{h,H}= \id_{M[H]}   
    & 
    (h\in H,\, H\leqslant G)  
    \\
 & \res{H}{L} \ind{K}{H} =   
 \sum_{x \in [L\backslash G / K]} \ind{L}{L\cap {}^x K} \, \con_{x ,L^x \cap K} \, \res{K}{L^x\cap K}
    &
     (L,K \leqslant H \leqslant G)  
\end{align*}
The last relation is the \emph{Mackey formula}.
A \emph{morphism} $\varphi: M\to N$ of Mackey functors is a family of $k$-linear maps $\varphi[H]: M[H]\to N[H]$ which commute with restriction, induction and conjugation maps in the evident way.

A \emph{\textup(commutative\textup) Green functor} is a Mackey functor~$R$ such that 
each~$R[H]$ carries the structure of a (commutative) associative unital ring, the restriction and conjugation maps are unital ring homomorphisms, and the following \emph{Frobenius formulas} hold:
\[
 \ind{L}{H}(\res{H}{L}(y)\cdot x) = y \cdot \ind{L}{H}(x)
 \quad , \quad
 \ind{L}{H}(x\cdot \res{H}{L}(y)) = \ind{L}{H}(x) \cdot y
\]
for all $L\leqslant H\leqslant G$, $x\in R[L]$ and $y\in R[H]$.
Similarly, a \emph{\textup(left\textup) Mackey module over~$R$} (or simply \emph{$R$-module}) is a Mackey functor $M$ where each $M[H]$ carries the structure of a (left) $R[H]$-module, in such a way that: 
\begin{align*}
& \res{H}{L}(r\cdot m) = \res{H}{L}(r)\cdot \res{H}{L}(m)  
 && (L\leqslant H\leqslant G , r\in R[H] , m\in M[H]) \\
& \con_{g,H}(r\cdot m) = \con_{g,H}(x)\cdot \con_{g,H}(m)
 && (g\in G, H\leqslant G,r\in R[H], m\in M[H])  \\
& r\cdot \ind{L}{H}(m) = \ind{L}{H}(\res{H}{L}(r)\cdot m)
 && (L\leqslant H\leqslant G, r\in R[H] , m\in M[L]) \\
& \ind{L}{H}(r)\cdot m = \ind{L}{H}(r\cdot \res{H}{L}(m)) 
 && (L\leqslant H\leqslant G, r\in R[L] , m\in M[H])
\end{align*}
A morphism of $R$-Mackey modules,  $\varphi \colon M\to M'$, is a morphism of the underlying Mackey functors such that each component $\varphi[H]$ is $R[H]$-linear.  
We will denote by 
\[
R\textrm- \Mack 
\]
the category of $R$-Mackey modules. 
We will see that it is a Grothendieck abelian category with a projective generator, and that it has a nice tensor product when $R$ is commutative.

\begin{example}
The \emph{Burnside ring} Green functor,~$R=Bur$, is defined by setting $Bur[H]:= K_0(H\textrm-\set)$, the Grothendieck ring of the category of finite $H$-sets with~$\sqcup$ and~$\times$  yielding sum and multiplication, and with the structure maps induced by the usual restriction, induction and conjugation operations for $H$-sets.
It turns out that $Bur$ acts uniquely on all Mackey functors, so that $Bur\textrm-\Mack$ is just~$\Mack$, the category of Mackey functors. (This is analogous to $\Z\textrm-\Mod=\Ab$).
\end{example}

\begin{remarks}

Instead of using abelian groups for the base category, it is often useful in applications to allow more general abelian categories, such as modules over some base commutative ring~$k$, possibly graded. 
It is straightforward to adapt the definitions. 
For our applications, it will sometimes be useful to let our Mackey functors take values in the category of $\Z/2$-graded abelian groups and degree preserving homomorphisms.
(A similar remark holds for the two other pictures.)
\end{remarks}

\subsection{The $G$-set picture}
The second picture is in terms of ``bifunctors'' on the category of finite $G$-sets. 
Now a Mackey functor is defined to be a pair of functors $M=(M^\star,M_\star)$ from $G$-sets to abelian groups, with $M^\star$ contravariant and $M_\star$ covariant, having the same values on objects: $M^\star(X)=M_\star(X)=: M(X)$ for all $X\in G\textrm-\set$.
Moreover, two axioms have to be satisfied:
\begin{enumerate}

\item $M$ sends every coproduct $X\to X\sqcup Y \gets Y$ to a direct-sum diagram in~$\Ab$.

\item $M^\star(g)M_\star(f)=M_\star(f')M^\star(g')$ for every pull-back square 
$
\xymatrix@C=10pt @R=10pt{\cdot \ar[d]_{g'} \ar[r]^-{f'} & \cdot \ar[d]^g \\ \cdot \ar[r]^-f & \cdot }
$ 
in~$G\textrm-\set$.
\end{enumerate}
Morphisms are natural transformations $\varphi=\{\varphi(X)\}_X$, where naturality is required with respect to both functorialities.  
Every Mackey functor in this new sense determines a unique Mackey functor in the previous sense, by setting
\[
M[H] := M(G/H)
\]
and
$\res{H}{L}:= M^\star (G/L\twoheadrightarrow G/H)$, 
$\ind{L}{H}:= M_\star(G/L\twoheadrightarrow G/H)$ and
$\con_{g,H}:= M^\star({}^g\! H\cong H)= M_\star(H\cong {}^g\! H)$.
Conversely, by decomposing each $G$-set into orbits we see how a Mackey functor in the old sense determines an (up to isomorphism, unique) Mackey functor in the new sense.


\subsection{The functorial picture and the Burnside-Bouc category $\mathcal B_R$}
\label{subsec:functorial}
Since Lindner~\cite{lindner:remark}, it is known that one can ``push'' the two functorialities of Mackey functors into the domain category, so that Mackey functors are -- as their name would suggest --  just ordinary (additive) functors on a suitable category.
It was proved by Serge Bouc that a similar trick can be performed also for Mackey modules, as follows
(see\footnote{
Beware that we prefer to use the opposite category, thinking of presheaves, so that Bouc's original notation $\mathcal C_A$ denotes the same category as our~\((\mathcal B_R)^\op\) (his $A$ being the Green functor~$R$). This is rather immaterial though, in view of the isomorphism $\mathcal B_R \cong (\mathcal B_R)^\op$
} 
\cite{bouc:green_functors}*{\S3.2}).

For any Mackey functor~$M$ and any finite $G$-set~$X$, let~$M_X$ be the Mackey functor which, in the $G$-set picture, is given by 
\[
M_X(Y):= M(Y \times X) 
\qquad (Y\in G\textrm-\set).
\]

Let~$R$  be a Green functor. 
If $M$ is an $R$-module, then $M_X$ inherits a natural structure of $R$-module, and the assignment $M\mapsto M_X$ extends to an endofunctor on $R\textrm-\Mack$ which is its own right and left adjoint  (\cite{bouc:green_functors}*{Lemma~3.1.1}).

 By \cite{bouc:green_functors}*{Prop.~3.1.3}, there is an isomorphism
\begin{align}
\label{bouc_yoneda}
\alpha_{X,M}  
\colon R\textrm-\Mack (R_X, M) \cong M(X) 
\end{align}
natural in $X\in G\textrm-\set$ and $M\in R\textrm-\Mack$. 
This looks suspiciously like the Yoneda lemma. In fact, it can be \emph{turned} into the Yoneda lemma!
It suffices to define an (essentially small $\Z$-linear) category~$\mathcal B_R$ as follows. 
Its objects are the finite $G$-sets, and its morphism groups are defined by $\mathcal B_R(X, Y):= R(X\times Y)$. The composition of morphisms in~$\mathcal B_R$ is induced by that of the category of $R$-Mackey modules, via the natural bijection~$\alpha_{X,M}$. 
The resulting embedding $\mathcal B_R\to R\textrm-\Mack$, $X\mapsto R_X$, extends along the (additive) Yoneda embedding $\mathcal B_R \to \Ab^{(\mathcal B_R)^\op}$, $X\mapsto \mathcal B_R(\variable,X)$, to an equivalence of categories  (\cite{bouc:green_functors}*{Theorem 3.3.5})
\[
\Ab^{(\mathcal B_R)^\op} \simeq R\textrm-\Mack \,.
\] 
Thus the functor $\mathcal B_R\to R\textrm-\Mack$ sending~$X$ to~$R_X$ is identified with the Yoneda embedding, and \eqref{bouc_yoneda} with the Yoneda lemma.
In particular, the category of Mackey modules over~$R$ is an abelian functor category, and we see that the representables  $R_{G/H}$ ($H\leqslant G$) furnish a finite set of projective generators.

\begin{example}
\label{ex:Bur}
For the Burnside ring $R=Bur$, the category~$\mathcal B_{Bur}$ is just the \emph{Burnside category}~$\mathcal B$, which has finite $G$-sets for objects, Hom sets $\mathcal B(X,Y) = K_0($G$\textrm-\set \downarrow X\times Y) $, and composition induced by the pull-back of $G$-sets. 
We recover this way Lindner's picture $\Mack\simeq \Ab^{\mathcal B^{\op}}$ of Mackey functors.
The product $X\times Y$ of $G$-sets clearly provides a tensor product (\emph{i.e.}, a symmetric monoidal structure) on~$\mathcal B$ with unit object~$G/G$. By the theory of Kan extensions (\emph{i.e.}, ``Day convolution''~\cite{day}), there is, up to canonical isomorphism, a unique closed symmetric monoidal structure on the presheaf category $\Ab^{\mathcal B^{\op}}$ which makes the Yoneda embedding $\mathcal B \hookrightarrow \Ab^{\mathcal B^\op}$ a symmetric monoidal functor. 
This is usually called the \emph{box product} of Mackey functors and is denoted by~$\boxpr$. 
It turns out that a Green functor is quite simply a monoid ($=$ ring object) in the tensor category $(\Mack, \boxpr, Bur)$, and it follows that one can study the whole subject of Green functors and Mackey modules from the categorical point of view; it is the fruitful approach taken by L.\,G.~Lewis~\cite{lewis:mimeo}.
\end{example}

\subsection{The tensor abelian category of $R$-Mackey modules}
\label{subsec:tensor}
If we consider a commutative Green functor~$R$ to be a commutative monoid in $(\Mack, \boxpr, Bur)$, as in Example~\ref{ex:Bur}, then the tensor product $M \,\boxpr_R\, N$ of two $R$-modules $M$ and~$N$
with structure maps $\rho_M: R \,\boxpr\, M\to M$ and $\rho_N: R \,\boxpr\, N\to N$, respectively, is defined by the following coequalizer in~$\Mack$
\[
\xymatrix{
M \,\boxpr\, R \,\boxpr\, N \ar@<0.5ex>[rr]^-{M\,\boxpr\, \rho_N} 
 \ar@<-0.5ex>[rr]_-{(\rho_M \circ \gamma) \,\boxpr\, N} && 
 M \,\boxpr\, N \ar[r] &
 M \,\boxpr_R\, N
} \,,
\]
where $\gamma$ denotes the symmetry isomorphism of the box product. 
Concretely, the value of $M\boxpr_R N$ at a $G$-set~$X$ is the quotient
\[
(M \,\boxpr_R\, N)(X) = \left( \bigoplus_{\alpha \colon Y\to X} M(Y)\otimes_\Z M(Y) \right) / \mathcal J ,
\]
where the sum is over all $G$-maps into~$X$, and where $\mathcal J$ is the subgroup generated by the elements
\begin{align*}
M_\star(f)(m)\otimes n' - m\otimes N^\star(f)(n') 
\quad & , \quad
M^\star(f)(m')\otimes n - m' \otimes N_\star(f)(n) \;, \\
m \cdot r \otimes n & - m \otimes r\cdot n
\end{align*}
for all $r\in R(Y)$, $m\in M(Y)$, $m' \in M(Y')$, $n\in N(Y)$, $n' \in N(Y')$
 and all morphisms $f\colon (Y, \alpha)\to (Y', \alpha')$ in the slice category $G\textrm- \set \!\downarrow\! X$, \ie, all $G$-maps $f\colon Y\to Y'$ such that $\alpha'\circ f= \alpha$
(see \cite{bouc:green_functors}*{\S6.6}).
 
As usual, this extends to define a closed symmetric monoidal structure on $R\textrm-\Mack$ with unit object~$R$. 
The internal Hom functor 
$\inthom_R(\variable,\variable) \colon (R\textrm-\Mack)^\op \times R\textrm-\Mack\to R\textrm-\Mack$, which of course is characterized by the natural isomorphism
\begin{align}
\label{inthom_adj}
R\textrm-\Mack (M\boxpr_R N , L) \cong R\textrm-\Mack(M, \inthom_R(N,L)) \,,
\end{align}
has also the following more concrete, and rather useful, description:
\begin{align}
\label{inthom}
\inthom_R (M,N)(X) = R\textrm- \Mack(M,N_X)
\end{align}
for every $G$-set~$X$ (see \cite{bouc:green_functors}*{Prop.~6.5.4}).

Finally, the tensor product extends to \emph{graded} $R$-Mackey modules $M$ and~$N$ by the familiar formula
\[
(M \,\boxpr_R\, N)_\ell := \bigoplus_{i+j = \ell} M_i \,\boxpr_R\, N_j \,.
\]
We will consider grading by an infinite or finite cyclic group~$\Z/\pi$ ($\pi\in\N$), \cf~\S\ref{sec:rel_hom}. 

\begin{remark}
\label{rem:alt_def_tensor}
It follows from the natural isomorphism $R_X\boxpr_R R_Y \cong R_{X\times Y}$ (see \cite{lewis:mimeo}*{Prop.~2.5}) that the tensor product restricts to representable modules in the functorial picture $\smash{R\textrm-\Mack \simeq \Ab^{(\mathcal B_R)^\op} }$, inducing a tensor product on $\mathcal B_R$ which is simply $X\times Y$ on objects. Therefore, we may recover~$\boxpr_R$ as the Day convolution product extending the tensor structure of~$\mathcal B_R$ back to all $R$-modules.
\end{remark}

\subsection{Induction and restriction of Mackey functors}
\label{subsec:ind_res_mac}

Just for a moment, let us see what happens if we allow the group~$G$ to vary. 
Given a Mackey functor~$M$ for~$G$ and a subgroup~$G'\leqslant G$, there is an evident restricted Mackey functor for~$G'$, written $\Res{G}{G'}(M)$, which is simply $\Res{G}{G'}(M)[H]:= M[H]$ at each $H\leqslant G'$. 
We obtain this way a functor $\Res{G}{G'}$ from the category of Mackey functors for~$G$ to the category of Mackey functors for~$G'$.
In particular, we see that if $R$ is a Green functor for~$G$ then $\Res{G}{G'}(R)$ is a Green functor for~$G'$, and that restriction may be considered as a functor $\Res{G}{G'}\colon R\textrm- \Mack \to \Res{G}{G'}(R)\textrm- \Mack$. 

Interestingly, there is also an induction functor $\Ind{G'}{G}$ going the opposite way which is both left and right adjoint to $\Res{G}{G'}$ (see \cite{bouc:green_functors}*{$\S$8.7}). 
It can be constructed as follows: given a $\Res{G}{G'}(R)$-module~$M$ and an $H\leqslant G$, set 
\begin{equation}
\label{ind_subgps}
\Ind{G'}{G}(M)[H]
:= \bigoplus_{a\in [G'\backslash G/H]} M[G'\cap {}^aH].
\end{equation}
Each summand is made into an $R[H]$-module in the evident way, that is, via the composite ring homomorphism 
$
\con_{a,G'^a\cap H} \res{H}{G'^a\cap H}\colon R[H]\to R[G'\cap {}^a H].
$
In the subgroup picture of Mackey functors, we have the following simple formulas:
\[
\Ind{G'}{G}(M)(X)= M(\Res{G}{G'} X)
\quad , \quad
\Res{G}{G'}(N)(Y)=N(\Ind{G'}{G} Y)
\]
for all $G$-sets~$M$ and $G'$-sets~$N$. 

These restriction and induction functors for Mackey modules satisfy ``higher versions'' of the expected relations. 
For instance, there is a Mackey formula isomorphism (\cite{thevenaz-webb:structure}*{Prop.~5.3}), as well as the following Frobenius isomorphism (see also \cite{bouc:green_functors}*{\S10.1} for more general results of this type).

\begin{prop}
\label{prop:higher_frob}
There is a natural isomorphism of $R$-Mackey modules
\begin{equation*}
\Ind{G'}{G}(M) \,\boxpr_R\, N 
\cong \Ind{G'}{G} \big( M \,\boxpr_{\Res{G}{G'}(R)} \, \Res{G}{G'}(N) \big)
\end{equation*}
for all $N\in R\textrm- \Mack$ and~$M\in \Res{G}{G'}(R)\textrm- \Mack$.
\end{prop}

\begin{proof}
We will use for this proof the $G$-set picture of Mackey functors.
Since there is no ambiguity, we will drop the decorations on all induction and restriction functors in order to avoid clutter.
Let us start -- innocently enough -- with a much more evident Frobenius isomorphism, namely, the natural isomorphism of $G$-sets
\[
\Ind{}{}(X) \times  Y \cong \Ind{}{} (X \times \Res{}{}Y)
\]
that exists for all $G'$-sets~$X$ and all $G$-sets~$Y$. 
It follows from this that, for an arbitrary $L\in R\textrm-\Mack$, we may identify 
\begin{equation}
\label{res_augm}
\Res{}{}(L_Y)\cong \Res{}{}(L)_{\Res{}{} Y} 
\end{equation}
because of the computation
\begin{align*}
\Res{}{}(L_Y)(X) 
&= L_Y (\Ind{}{} X) \\
&= L((\Ind{}{}X) \times Y) \\
&\cong L(\Ind{}{}(X \times \Res{}{} Y)) \\
&= \Res{}{} L(X \times \Res{}{}Y) \\
&= (\Res{}{}L)_{\Res{}{} Y}(X).
\end{align*}
Next, we claim the existence of a natural isomorphism
\begin{equation}
\label{crucial_ident_hom}
\Ind{}{} \inthom_{\Res{}{}(R)} (M , \Res{}{} L ) 
\cong \inthom_{R} (\Ind{}{} M ,  L )
\end{equation}
for all $M$ and~$L$.
Indeed, evaluating at every $Y\in G\textrm-\set$  we find
\begin{align*}
\Ind{}{} \inthom_{\Res{}{}(R)} (M , \Res{}{} L ) (Y)
& =  \inthom_{\Res{}{}(R)} (M , \Res{}{} L ) (\Res{}{} Y) \\
& = \Res{}{}(R)\textrm- \Mack ( M , \Res{}{} (L)_{\Res{}{} Y} ) \\
& \cong \Res{}{}(R)\textrm- \Mack ( M , \Res{}{} (L_Y) ) \\
& \cong R\textrm- \Mack ( \Ind{}{} M, L_Y ) \\
& = \inthom_R( \Ind{}{} M, L)(Y)
\end{align*}
by using~\eqref{inthom} in the second and in the last lines, \eqref{res_augm} in the third line, and the $(\Ind{}{}, \Res{}{})$-adjunction in the fourth.
Finally, there is a natural isomorphism
\begin{align*}
R\textrm-\Mack \big( \Ind{}{} (M \boxpr_{\Res{}{}(R)} \Res{}{} N), L \big) 
& \cong \Res{}{}(R)\textrm- \Mack \big( M\boxpr_{\Res{}{}(R)} \Res{}{}N, \Res{}{} L \big) \\
& \cong \Res{}{}(R) \textrm- \Mack \big( \Res{}{} N , \inthom_{\Res{}{}(R)} (M , \Res{}{} L )  \big) \\
&\cong R\textrm- \Mack \big( N , \Ind{}{} \inthom_{\Res{}{}(R)} (M , \Res{}{} L ) \big) \\
& \cong R\textrm- \Mack \big( N ,  \inthom_{R} (\Ind{}{} M ,  L ) \big) \\
& \cong R\textrm- \Mack \big( (\Ind{}{}M ) \boxpr_R N ,   L \big)
\end{align*}
by consecutive application of the $(\Ind{}{}, \Res{}{})$-adjunction, the $(\boxpr, \inthom)$-adjunction \eqref{inthom_adj}, the $(\Ind{}{}, \Res{}{})$-adjunction once again, the isomorphism~\eqref{crucial_ident_hom}, and the other $(\boxpr, \inthom)$-adjunc\-tion.
Since this isomorphism is natural in~$L$ and since~$L$ is an arbitrary $R$-module, we conclude by Yoneda the existence of a natural isomorphism
\begin{equation*}
\Ind{}{}(M)\boxpr_{R} N 
\cong \Ind{}{} \left( M \boxpr_{\Res{}{}(R)} \Res{}{}(N) \right)
\end{equation*}
of $R$-modules.
\end{proof}

\section{Equivariant K-theory as a Mackey module}
\label{sec:K_mack}

\subsection{The representation Green functor}

Let us describe the commutative Green functor that will concern us here, the \emph{representation Green functor}, that we denote~$R^G$. 
It is also one of the most classical examples. 


By definition, the value $R^G[H]$ at the subgroup $H\leqslant G$ is the complex representation ring~$R(H):=K_0(\C H\textrm-\mod)$.  
Addition is induced by the direct sum of modules and multiplication by their tensor product over the base field~$\C$, equipped with the diagonal $G$ action. 
For $L\leqslant H\leqslant G$, the restriction maps $\res{H}{L}: R(H)\to R(L)$ are defined by restricting the action of a $\C H$-module to~$\C L$ via the inclusion $\C L\to \C H$, and the induction maps $\ind{L}{H}:R(L)\to R(H)$ are defined by
the usual induction of modules, $M\mapsto \C H\otimes_{\C L}M$ ($M\in \C L\textrm-\mod $).
The conjugation maps $\conj_{g,H}: R(H)\to R({}^gH)$, similarly to the restriction maps, are induced by precomposition with the isomorphisms 
$\C {}^g\!H\to \C H$, $x\mapsto g^{-1}xg$.
The verification that $R^G$ satisfies the axioms of a commutative Green functor is an easy exercise, and follows immediately from general text-book properties of modules over group rings (\emph{e.g.}~\cite{bensonI}*{$\S$3.3}).

\subsection{Equivariant K-theory}
\label{subsec:equiv_Kth}

For every separable $G$-$\Cstar$-algebra $A\in \Cstarsep^G$, we want to define a $\Z/2$-graded $R^G$-Mackey module 
\[
k^G_*(A):=\{ K^H_\epsilon(\Res{G}{H} (A)) \}^{H\leqslant G}_{\epsilon\in \Z/2}
\]
by collecting all its topological $K$-theory groups. 
In order to describe the structure maps of this $R^G$-module as concretely as possible, we now briefly recall from \cite{phillips:free}*{$\S$2} the definition of equivariant $K$-theory in terms of (Banach) modules.

Assume first that $A$ is unital.
A \emph{$(G,A)$-module}~$E$ consists of a right module $E$ over the ring~$A$, together with a representation $G\to \mathcal L(E)$ of~$G$ by continuous linear operators on~$E$, such that $g(ea)=(ge)(ga)$ for all $g\in G, e\in E, a\in A$. Of course, for $\mathcal L(E)$ to make sense, $E$ must be endowed with a topology; we do not belabor this point, because we will be exclusively concerned with modules that are projective and finitely generated over~$A$, and which therefore inherit a Banach space structure (and a unique topology) from that of~$A$.  

The direct sum of two $(G,A)$-modules is defined in the evident way with the diagonal $G$-action, and a morphism of $(E,A)$-modules is a continuous $A$-module map $\varphi: E\to E'$ commuting with the $G$-action: $\varphi(ge)=g\varphi(e)$ for all $g\in G, e\in E$.

Let $\tilde K^G_0(A)$ be the Grothendieck group of isomorphism classes of finitely generated $A$-projective $(G,A)$-modules, with addition induced by the direct sum.
If $V$ is a finite dimensional $\C G$-module and $E$ a $(G,A)$-module, we may equip the tensor product $V\otimes_\C E$ with the diagonal $G$-action $g(v\otimes e):=gv \otimes ge$ and the right $A$-action $(v\otimes e)a:=v\otimes ea$ ($g\in G, v\in V, e\in E, a\in A$),  thereby inducing a left $R(G)$-action on the abelian group~$\tilde K^G_0(A)$.  
The assignment $A\mapsto \tilde K^G_0$ becomes a covariant functor from unital $G$-$\Cstar$-algebras to $R(G)$-modules by extension of scalars; indeed, given a unital $G$-equivariant *-homo\-morph\-ism $f:A\to B$ and a $(G,A)$-module~$E$, we  equip the finitely generated projective $B$-module $E\otimes_AB$ with the $G$-action $g(e\otimes b):=ge\otimes gb$.

If $A$ is a general, possibly non unital, $G$-$\Cstar$-algebra, then by the usual trick we set $K^G_0(A):=\ker (\tilde K^G_0 (\pi_A: A^+\to \C ))$, where $\pi_A$ is the natural augmentation on the functorial unitization $A^+$ of~$A$. Then $K^G(A)= \tilde K^G(A)$ for unital~$A$, and $K^G_0$ yields a functor $\Cstaralg^G \to R(G)\textrm-\Mod$ on $G$-$\Cstar$-algebras.

\subsection{The $R^G$-Mackey module~$k^G(A)$}

We now define our K-theory Mackey module.  
For all $A\in \Cstaralg^G$ and $H\leqslant G$, set 
\[
k^G(A)[H]:= K^H_0(\Res{G}{H}(A)) \, .
\] 
For the definition of the structure maps, assume at first that $A$ is unital.

The restriction maps $\res{H}{L}$ are simply induced by restricting the $H$-action on $(H,A)$-modules to~$L$, as in the case of~$R^G$.
The conjugation maps are as follows. 
Given an $(H,\Res{G}{H}A)$-module~$E$, let $\Con_g(E)$ denote~$E$ equipped with the $H$- and $A$-actions
\[
(h \,{}^g\!\cdot e):= g^{-1}hge
\quad , \quad
(e \cdot^g a):= e (g^{-1}a) 
\quad\quad (e\in E, h\in H, g\in G).
\]

\begin{lemma}
\label{lemma:conj}
The above formulas proide a well-defined $({}^gH, \Res{G}{\, {}^g\! H}(A))$-module $\Con_g E$, and the assignment $E\mapsto \Con_g(E)$ induces a well-defined $R(H)$-linear homomorphism 
\[
\con_{g,H}\colon K^H_0(\Res{G}{H} A)\to K^{\,{}^g\! H}_0(\Res{G}{\,{}^g\! H} A).
\]
Moreover, 
$\con_{f,{}^gH} \, \con_{g,H} = \con_{fg, H}  $
for all  $f,g\in G, \, H\leqslant G$, 
and $\con_{g,H}=\id_{K^{H}_0(A)}$ whenever $g\in H$.
\end{lemma}

\begin{proof}
The two actions are certainly well-defined (to see this for the $A$-action, recall that~$G$ acts on~$A$ by algebra homomorphisms, which must be unital if~$A$ is unital), and they are compatible by the computation
\[
h \,{}^g\!\cdot (e \cdot^g a)
= (g^{-1}hg)(e \cdot g^{-1}a) 
= (g^{-1}hg e) (g^{-1}hg g^{-1} a)
= (h \,{}^g\!\cdot e) \cdot^g (h \cdot a) 
\]
($h\in H, a\in A, e\in E$).
Let $E$ be finitely generated projective over~$A$. But then $\Con_gE$ is also finitely generated projective, because (ignoring the group actions) the map $E\to \Con_gE$, $e\mapsto g^{-1}e$, is an $A$-linear isomorphism:
\[
g^{-1}(e a) = (g^{-1}e)(g^{-1}a) = (g^{-1} e) \cdot^g a 
\qquad (e\in E, a\in A). 
\]
The rest is similarly straightforward.
\end{proof}

\begin{remark}
\label{rem:conj_alg}
Perhaps a more natural way to understand the conjugation maps is to note that every $(H,A)$-module~$E$ can be considered as an $({}^gH, {}^gA)$-module, say~${}^gE$, where ${}^gA$ is the ${}^gH$-$\Cstar$-algebra with underlying $\Cstar$-algebra~$A$ and with the ${}^gH$-action $ghg^{-1}\,{}^g\!\! \cdot a = ha$, as in~\S\ref{subsec:res_ind_con}. 
This is just as for the restriction maps: both group actions, that on~$E$ and that on~$A$, are precomposed with a group homomorphism, in this case the conjugation isomorphisms ${}^gH\to H$, $h\mapsto g^{-1}hg$ (for restriction, the inclusion of a subgroup).
Similarly, we let ${}^gA$ act on ${}^gE$ simply by $e\cdot a= ea$, just as $A$ acted on~$E$, and the compatibility condition for ${}^gE$ is trivially satisfied because it is for~$E$. 
Now note that, if the $H$-action on~$A$ comes from an action of the whole group~$G$, then the $\Cstar$-algebra isomorphism $g^{-1}\colon A\smash{\stackrel{\sim}{\to}} \,{}^gA$ provided by the action is $G$-equivariant, since 
$g^{-1}(ha) 
= (g^{-1}hg)(g^{-1}a)
= h\,{}^g\!\cdot (g^{-1}a)$ for all $h\in G$ and $a\in A$.
Clearly, the restriction of ${}^gE$ along $g^{-1}$ is precisely the $({}^gH,A)$-module $\Con_gE$ defined above (or, with extension of scalars: ${(g^{-1})}_*(\Con_g E)\cong {}^gE$).
\end{remark}

We now define the induction maps, following~\cite{phillips:free}*{$\S$5.1}.
Let $L\leqslant H \leqslant G$.
If $E$ is an $(L,A)$-module, we define an $(H,A)$-module
\[
\Ind{L}{H}(E):=\{\varphi\colon H\to E \mid \varphi(x\ell)=\ell^{-1}\varphi(x) \;\;\forall \ell\in L, x\in H\}
\]
with the following $A$- and $H$-actions:
\[
(\varphi \cdot a)(x):= \varphi (x)(x^{-1}a)
\quad, \quad (h\cdot \varphi)(x):=\varphi(h^{-1}x)
\]
for all $\varphi\in \Ind{L}{H}(E)$, $a\in A$, and $x,h\in H$.
By \cite{phillips:free}*{Prop.\ 5.1.3}, the resulting functor $E\mapsto \Ind{L}{H}(E)$ from $(L,A)$-modules to $(H,A)$-modules preserves finitely generated projectives, and the induced homomorphism 
\[
\ind{L}{H}\colon K^L_0(\Res{G}{L}A)\to K^H_0(\Res{G}{H}A)
\]
is $R(H)$-linear. (Here as always, we turn $K^L_0(\Res{G}{L}A)$ into an $R(H)$-module via the ring homomorphism $\res{H}{L}\colon R(H)\to R(L)$.)

\begin{remark}
\label{remark:unit_identification}
Note that, when $A=\C$ is the trivial $G$-$\Cstar$-algebra, there are evident canonical isomorphisms $K^H_0(\C)\cong R(H)$ ($H\leqslant G$) that identify the respective induction, restriction and conjugation maps. In other words, we can identify $k^G(\C)=R^G$ as Mackey functors.
\end{remark}

As usual with $\Cstar$-algebras, it is easy to use the functorial unitisation to extend the definitions of $\res{H}{L}$, $\ind{L}{H}$ and $\con_{g,H}$ to general, possibly nonunital, algebras~$A$. For instance, $\ind{L}{H}$ is the map induced on kernels in the following morphism of short exact sequences:
\[
\xymatrix{
K^H_0(\Res{G}{H}A) \ar@{ >->}[r] &
  K^H_0(\Res{G}{H} A^+) \ar@{->>}[r] & 
   K^H_0(\Res{G}{H} \C)= R(H) \\
 K^H_0(\Res{G}{L}A) \ar@{ >->}[r] \ar@{..>}[u]^{\ind{L}{H}} &
  K^H_0(\Res{G}{L} A^+) \ar@{->>}[r] \ar[u]_{\ind{L}{H}} & 
   K^H_0(\Res{G}{L} \C)= R(L) \ar[u]_{\ind{L}{H}}  
}
\]
and similarly for $\res{H}{L}$ and $\con_{g,H}$. 
Because of the naturality of the definition, it will suffice to verify equalities between restriction, conjugation and induction maps for the case of unital algebras.

\begin{lemma}
\label{lemma:mackey_formula_kG}
There is an isomorphism of $(H,A)$-modules
\[
\Res{G}{H}\Ind{L}{G}(E)
\cong \bigoplus_{z\in [H\backslash G/L]}
\Ind{H\cap {}^z\! L}{H} \Con_z \Res{L}{H^z\cap L}(E)
\]
for every $(L,A)$-module $E$ and all subgroups $H,L\leqslant G$.
Moreover, once the set of representatives $[H\backslash G/L]$ is fixed, the isomorphism is natural in~$E$.
\end{lemma}

\begin{proof}
Every choice of the set $[H\backslash G/L]$ yields a basic decomposition
\[
{}_HG_L \cong \coprod_{z\in [H\backslash G/L]} HzL
\]
of $(H,L)$-bisets. There follows a decomposition of $(H,A)$-modules 
\begin{align*}
\Res{G}{H}\Ind{L}{G}(E) 
& = \left\{\varphi \colon \coprod_{z\in [H\backslash G/L]} HzL \to E 
        \mid \ell \varphi(x\ell)=\varphi(x) \; \forall x\in G, \ell\in L \right\} \\
& = \bigoplus_{z\in [H\backslash G/L]} 
   \underbrace{\left\{ \varphi \colon HzL \to E \mid  \ell \varphi(x\ell)=\varphi(x) \; \forall x\in HzL, \ell\in L  \right\}}_{
   \mbox{ \small $ =: V_z $}
   } \, .
\end{align*}
Of course the $H$-action on each summand~$V_z$ is still given by $(h\cdot \varphi)(x)=\varphi(h^{-1}x)$, and the $A$-action by $(\varphi\cdot a)(x)=\varphi (x)(x^{-1}a)$  (for all $x\in HzL, h\in H, a\in A$). 

For every~$z$, let $W_z:= \Ind{H\cap {}^z\! L}{H} \Con_z \Res{L}{H^z\cap L}(E)$ denote the corresponding summand of the right hand side of the Mackey formula. 
Here $\Con_z \Res{L}{H^z\cap L}(E) $ is $E$ equipped with the conjugated $H\cap  {}^z\! L$-action 
$h {}^z\! \cdot e = (z^{-1}h z) e$ 
(for $h\in H\cap {}^z\! L$ and $e\in E$), 
so that
\begin{align*}
W_z
& = \left\{\psi\colon H\to E \mid  (z^{-1}h z)\psi (yh) = \psi(y)  \;(y\in H , h\in H\cap{}^z\!L)  \right\} \\
& = \left\{\psi\colon H\to E \mid  \ell \psi (y z\ell z^{-1}) = \psi(y)  \;(y\in H , \ell \in H^z\cap L)  \right\} .
\end{align*}
On $W_z$ too the $H$-action is again $(h\cdot \psi)(y)=\psi(h^{-1}y)$, but now, because of conjugation, the $A$-action looks as follows:
\[
(\psi \cdot^z a)(y)= \psi(y)(z^{-1}y^{-1} a)  \quad \quad (\psi \in W_z, y\in H, a\in A).
\]
We claim that $V_z\cong  W_z$ via the function 
$ 
\varphi \mapsto \tilde\varphi
$
given by $\tilde\varphi(y):=\varphi(yz)$ for all $y\in H$.
The function is well-defined, because $yz\in HzL$ for all $y\in H$ and
\[
\ell \cdot \tilde\varphi(y z\ell z^{-1}) 
= \ell \cdot \varphi (yz\ell) 
= \ell \ell^{-1} \varphi(yz) 
= \tilde\varphi(y)
\]
for all $\ell\in L$. 
It is also evidently $H$-linear, and it is $A$-linear by the computation
\[
(\widetilde{\varphi \cdot a})(y) 
=  (\varphi \cdot a)(yz) 
= \varphi (yz ) (z^{-1}y^{-1}a)
= (\tilde\varphi \cdot^z a)(y)
\]
($\varphi\in V_z, a\in A, y\in H$).
Finally, we claim that the inverse map $\psi \mapsto \hat\psi$, $W_z\to V_z$, is given by the formula $\hat \psi (x):= \ell^{-1}  \psi (h)$ for each $x=hz\ell \in Hz L$.
The map~$\hat \psi$ is well-defined: if $x= hz\ell = h_1 z \ell_1 \in HzL$ and $\psi\in W_z$, then
\begin{align*}
\ell^{-1}\psi(h) 
& = \ell^{-1} \psi (h_1z\ell_1\ell^{-1} z^{-1})
= \ell^{-1} \psi (h_1z\ell_1 z^{-1}(z\ell^{-1} z^{-1}))
 & \\
&= \psi (h_1z\ell_1 z^{-1})    
 & (\psi\in W_z) \\
&= \ell_1^{-1}\psi( h_1 )
 & (\psi\in W_z)
\end{align*}
Moreover, the computation (with $x= hz\ell \in HzL, \ell' \in L$)
\[
\ell'  \hat \psi (x \ell') 
=  \ell' (\ell \ell')^{-1} \psi(h) 
= \ell^{-1} \psi(h)
 = \hat\psi(x)
\]
shows that indeed $\hat\psi \in V_z$ for all $\psi \in W_z$.
The verification that
$ (\hat\psi)^{\sim }   = \psi $
 and 
$(\tilde \varphi ){\,\hat {}} = \varphi$
 is equally immediate:
 \begin{align*}
 & (\tilde \varphi ){\,\hat {}} 
 = \ell^{-1} \tilde \varphi(h)
 =\ell^{-1} \varphi(hz)
 = \varphi(hz\ell)
 = \varphi(x)
 & 
  (\varphi\in V_z, x=hz\ell \in HzL), \\
 & (\hat\psi)^{\sim }(y)  
 = \hat \psi (yz)
 = \psi (y)
 &   (\psi\in W_z, y\in H).
 \end{align*}
  Hence we obtain the claimed  isomorphism $V_z\cong W_z$ of $(H,A)$-modules.
Therefore we have an isomorphism as claimed in the lemma, and it is evident from its construction that it  is natural in the $(L,A)$-module~$E$.
\end{proof}

\begin{prop}
\label{prop:mackey_Kth}
The modules $K^H_0(A)$ and the maps $\res{H}{L}$, $\ind{L}{H}$ and $\con_{g,H}$ described above define an $R^G$-Mackey module~$k^G(A)$, and the functorialities of all~$K^H_0$ assemble to yield a functor 
$k^G\colon \Cstaralg^G\to R^G\textrm- \Mack$, with $k^G(\unit)=R^G$.
\end{prop}

\begin{proof}
First of all, let us fix a $G$-$\Cstar$-algebra~$A$, and 
let us verify that $k^G(A)$ is a Mackey functor. 
The Mackey formula holds by Lemma~\ref{lemma:mackey_formula_kG}.
The first and fifth relations (see $\S$\ref{subsec:mackey}) are either contained in Lemma~\ref{lemma:conj} or in \cite{phillips:free}*{Prop.\ 5.1.3}.
There remain the compatibilities of the conjugation maps with restrictions and inductions, which are straightforward and are left to the reader.
Thus $k^G(A)$ is a Mackey functor.
Next, we must verify that the collected actions of $R(H)$ on~$K^G_0(\Res{G}{H} A)$ turn $k^G(A)$ into an $R^G$-module.
The third and fourth relations for Mackey modules are proved in~\cite{phillips:free}*{Prop.\ 5.1.3}
(the third one under the guise of the $R(H)$-linearity of $\ind{L}{H}$).
The $R^G$-linearity of restriction and conjugation maps (first and second relations), are easier and are left to the reader.
Finally, the functoriality of $A\mapsto k^G(A)$ for equivariant $*$-homomorphisms follows immediately from that of each $H$-equivariant $K$-theory, and we have already seen (Remark~\ref{remark:unit_identification}) that $k^G(\unit)=R^G$.
\end{proof}

\begin{remark}
As usual we set $K^G_1(A):=K^G_0( A[1])$, so that we get a functor 
\[
K^G_* \colon \Cstaralg^G\to R(G) \textrm- \Mod_{\Z/2}
\quad , \quad
K^G_*(A):=\{K^G_\epsilon(A)\}_{\varepsilon\in \Z/2}
\] 
to graded modules and degree preserving morphisms.
We similarly obtain a functor 
\[
k^G_* \colon \Cstaralg^G\to R(G)\textrm- \Mack_{\Z/2}
\quad , \quad
k^G_*(A):=\{K^H_\epsilon(A)\}^{H\leqslant G}_{\varepsilon\in \Z/2}
\] 
into the category of $\Z/2$-graded Mackey modules over~$R^G$. Alternatively, the target category of $k^G_*$ may be understood as the category of Mackey modules over~$R^G$ based in the category of $\Z/2$-graded abelian groups.
\end{remark}

\subsection{The extension to the Kasparov category}
\label{subsec:ext}
Let us restrict our attention to the the subcategory of separable $G$-$\Cstar$-algebras, $\Cstarsep^G$. 
Next, we extend our functor~$k^G$ to the $G$-equivariant Kasparov category and study the properties of the extension. 

\begin{lemma}
\label{lemma:mackey_KKth}
The functor $k^G$ has a unique lifting,  that we also denote by~$k^G$, to the Kasparov category~$\KK^G$ along the canonical functor $\Cstarsep^G\to \KK^G$.
\end{lemma}

\begin{proof}
By the universal property of the canonical functor $\Cstarsep^G\to \KK^G$, as proved in~\cite{meyer_homom}, the existence and uniqueness of such a lifting is equivalent to the functor $k^G$ being homotopy invariant, $\Cstar$-stable and split exact (in the $G$-equivariant sense).
This follows immediately from the basic fact that each $K$-theory functor $K^H_0\circ \Res{G}{H}$ does enjoy the three properties.
\end{proof}

Note that, for all $H\leqslant G$, the Green functor $R^H$ is just the restriction of $R^G$ at~$H$, in the sense of~$\S$\ref{subsec:ind_res_mac}: $R^H=\Res{G}{H}(R^G)$. 
Therefore, as explained there, the evident restriction functor $\Res{G}{H}\colon R^G\textrm- \Mack\to R^H \textrm- \Mack$ has a left-and-right adjoint~$\Ind{H}{G}$.

\begin{lemma}
\label{lemma:Ind_k}
For all $H\leqslant G$, the diagrams 
\[
\xymatrix{
\KK^G_{\phantom{G}} \ar[r]^-{k^G} \ar[d]_{\Res{G}{H}} & R^G\textrm- \Mack \ar[d]^{\Res{G}{H}} \\
\KK^H_{\phantom{H}} \ar[r]^-{k^H} & R^H\textrm- \Mack
}
\quad\quad
\xymatrix{
\KK^G_{\phantom{G}} \ar[r]^-{k^G} & R^G\textrm- \Mack \\
\KK^H_{\phantom{H}} \ar[r]^-{k^H} \ar[u]^{\Ind{H}{G}} & R^H\textrm- \Mack \ar[u]_{\Ind{H}{G}}
}
\]
commute up to isomorphism of functors.
\end{lemma}

\begin{proof}
The claim involving the restriction functors is evident from the definitions; in this case, the square even commutes strictly. 
Now we prove the claim for induction.

Let $A\in \KK^H$ and $L\leqslant G$. In the case of the rank-one free module, the Mackey formula of Lemma~\ref{lemma:mackey_formula_kG} can be easily rewritten as the following isomorphism of $L$-$\Cstar$-algebras:
\[
\Res{G}{L}\Ind{H}{G}(A) 
\cong
 \bigoplus_{x\in [L\backslash G/H]} \Ind{L\cap {}^xH}{L} \Res{{}^xH}{L\cap {}^xH}(\,{}^x\! A) \,.
\]
Once we have fixed the set of representatives for $L\backslash G/H$, the isomorphism becomes natural in~$A$.
Therefore we get a natural isomorphism
\begin{align*}
(k^G \circ \Ind{H}{G}(A))[L]
& = K^L_0\left( \Res{G}{L} \Ind{H}{G}(A) \right)  \\
& \cong \bigoplus_{x\in [L\backslash G/H]} K^L_0\left(  \Ind{L\cap {}^xH}{L}  \Res{{}^xH}{L\cap {}^x H}( {}^x\! A)  \right) \\
& \cong  \bigoplus_{x\in [L\backslash G/H]} K^{L\cap{}^xH}_0\left( \Res{{}^xH}{L\cap {}^x H}(  {}^x\! A)  \right) \\
& \cong   \bigoplus_{x\in [L\backslash G/ H]} k^{\,{}^x H}( {}^x A)[L\cap {}^xH]  \\
& \cong   \bigoplus_{x\in [L\backslash G/ H]} k^H(A)[H\cap L^x]  \\
& = (\Ind{H}{G}\circ k^H(A))[L] \, .
\end{align*}
In the third line we have used the  $(\Ind{}{},\Res{}{})$-adjunction for  Kasparov theory, and in the fifth we have used the $H,{}^x\!H$-equivariant isomorphism $A\cong {}^xA$ of Remark~\ref{rem:conj_alg} and the isomorphism $K^H_0(\Res{H}{L^x\cap H}A)\cong K^{\,{}^x\! H}_0(\Res{\,{}^x\!H}{L\cap \,{}^x\!H} {}^xA)$ it induces; the last line is~\eqref{ind_subgps} with $a=x^{-1}$.
This proves the claim.
\end{proof}

The next theorem is the main result of this article.

\begin{thm}
\label{thm:fully_faithful}
The restriction of $k^G:\KK^G \to R^G\textrm-\Mack$ to the full subcategory $\{C(G/H)\colon H\leqslant G\}$ of $\KK^G$  is a fully faithful functor.
\end{thm}

\begin{proof}
Identifying $k^G\circ \Ind{}{} = \Ind{}{} \circ k^H$ and $k^H\circ \Res{}{}= \Res{}{} \circ k^G$ as in Lemma~\ref{lemma:Ind_k}, for all $H,L\leqslant G$ we have the following commutative diagram:
\begin{align*}
\xymatrix{
\KK^G(C(G/H), C(G/L))  \ar[r]^-{k^G} \ar@{=}[d] &
 R^G\textrm-\Mack (k^G C(G/H), k^G C(G/L)) \ar@{=}[d] \\
\KK^G(\Ind{H}{G} \unit , \Ind{L}{G} \unit) \ar[r]^-{k^G} \ar[d]_{\cong} & 
 R^G\textrm-\Mack (\Ind{H}{G} k^H(\unit), \Ind{L}{G} k^L(\unit)) \ar[d]^{\cong} \\
\KK^H(\unit , \Res{G}{H}\Ind{L}{G} \unit) \ar[r]^-{k^G} \ar[d]_{\cong} & 
 R^H\textrm-\Mack (R^H, k^H(\Res{G}{H}\Ind{L}{G} \unit)) \ar[d]_{\can}^{\cong} \\
 K^H_0(\Res{G}{H}\Ind{L}{G} \unit) \ar[r]^-{\can}_-{\cong} &
  R(H)\textrm- \Mod (R(H), K^H_0(\Res{G}{H}\Ind{L}{G} \unit))
}
\end{align*}
Therefore the upper map labeled $k^G$ is bijective.
\end{proof}

\subsection{The Burnside-Bouc category as equivariant KK-theory}
\label{subsec:BBKK}

To complete the picture, we can now describe the  Burnside-Bouc category associated with the representation ring~$R^G$ in terms of $G$-equivariant Kasparov theory. The relation is a very simple and satisfying one.

\begin{defi}
In analogy with permutation modules, we call a $G$-$\Cstar$-algebra of the form $C_0(X)$, for some $G$-set~$X$, a \emph{permutation algebra}. 
Let 
\[
\mathsf{Perm}^G  \quad \textrm{ resp. } \quad \mathsf{perm}^G
\] be the full subcategory of~$\KK^G$ of separable permutation algebras, respectively of finite dimensional permutation algebras. Note that they are precisely those of the form $A\cong \bigoplus_{i\in I} C(G/H_i)$ for some countable, respectively finite, index set~$I$. Note also that, by virtue of the natural isomorphisms $C(X)\oplus C(Y)\cong C(X\sqcup Y)$ and $C(X)\otimes C(Y)\cong C(X\times Y)$, both $\mathsf{Perm}^G$ and $\mathsf{perm}^G$ are additive tensor subcategories of~$\KK^G$.
\end{defi}

\begin{thm}
\label{thm:burnside_bouc_R}
For every finite group~$G$, the functor $k^G: \KK^G \to R^G\textrm-\Mac$ of~\S\ref{subsec:ext} restricts to a tensor equivalence of 
$\mathsf{perm}^G$ with the full subcategory of representable $R^G$-modules, \emph{i.e.}, with the Burnside-Bouc category~$\mathcal B_R$.
\end{thm}

\begin{proof}
Identifying $R^G\textrm-\Mack = \Ab^{(\mathcal B_{R^G})^\op}$ as in~\S\ref{subsec:functorial}, we obtain the following diagram of functors, which we claim is commutative (up to isomorphism).
\begin{equation*}
\xymatrix{
(G\textrm-\set)^\op
  \ar[rr]^-{\can}
    \ar[dd]_{C} 
     \ar[dr]^-{C} &&
 \mathcal B  \ar[dd]|{X\mapsto R^G_X} 
  \ar[r]^-{\can} &
   \mathcal B_{R^G} \ar[ddl]^-{\textrm{Yoneda}}
   \\
 & \Cstarsep^G
  \ar[dl]_{\can}
   \ar[dr]^-{k^G} 
    \ar@{}[ur]|\beta_{\cong}
     && \\
 \KK^G \ar[rr]^-{k^G} && R^G \textrm- \Mack &
}
\end{equation*}
Indeed, the left, bottom and right triangles (strictly) commute by definition.
We must show that there is a natural isomorphism $\beta_X\colon k^G \circ C(X) \cong R^G_{X}$, making the central triangle commute. 
For $X=G/H$ and $Y=G/L \in G\textrm-\set$, we obtain the following isomorphisms~$\beta_{G/H}(G/L)$:
\begin{align*}
 k^G (C(G/H))(G/L) 
& =  K^L_0(\Res{G}{L} C(G/H)) \\
& =  K^L_0(C( \Res{G}{L}  G/H)) \\
& \cong K^L_0 C\left( \coprod_{x\in [L\backslash G/H]} L/(L \cap {}^x H) \right)  \\
& \cong \bigoplus_{x\in [L\backslash G/H]} K^L_0 C(L/(L \cap {}^xH)) \\ 
& \cong \bigoplus_{x\in [L\backslash G/H]} R(L \cap {}^x H) \\
& \cong R^G \left(  \bigoplus_{x\in [L\backslash G/H]} G/(L \cap {}^x H) \right) \\
& =   R^G ( G/L \times G/H) \\
& = R^G_{G/H} (G/L)
\end{align*}
We leave to the reader the verification that, by letting $L\leqslant G$ vary, these define an isomorphism $\beta_{G/H}\colon k^GC(G/H)\cong R^G_{G/H}$, and that the latter can be extended to a natural isomorphism $\beta_X$ as required.

The statement of the theorem follows now from the fact that the bottom horizontal~$k^G$ is fully faithful on the image of~$C$, by Theorem~\ref{thm:fully_faithful}; the ``tensor'' part follows from the identification $k^G(\C)= R^G$ and from the natural isomorphism
\begin{align*}
\phi_{X,Y} \quad \colon \quad
k^G(C(X)) \boxpr_{R^G} k^G (C(Y))
&\cong R^G_X \boxpr_{R^G} R^G_Y \\
&\cong  R^G_{X\times Y} \\
&\cong k^G(C(X\times Y)) \\
&\cong k^G(C(X)\otimes C(Y)) 
\end{align*}
for all $X,Y\in G\textrm-\set$, obtained by combining~$\beta$ with the symmetric monoidal structures of the functor~$C$ and of the Yoneda embedding $X\mapsto R_X$.
Clearly the square
\[
\xymatrix{
k^GC(X) \boxpr_{R_G} k^GC(Y)
 \ar[d]_{\phi_{X,Y}} \ar[rr]^-{k^G(\textrm{switch})} &&  k^GC(Y) \boxpr_{R_G} k^GC(X) \ar[d]^{\phi_{Y,X}}
 \\
k^G(C(X)\otimes C(Y)) \ar[rr]^{\textrm{switch}} && 
  k^G(C(Y)\otimes C(X))
}
\]
is commutative, showing that~$\phi$ turns~$k^G$ into a symmetric monoidal functor on the image of~$C$.
\end{proof}

\begin{cor}
\label{cor:ident_mackey_modules}
The category of additive functors $(\mathsf{perm}^G)^\op\to \Ab$ is equivalent to the category of Mackey modules over the representation Green functor~$R^G$.
If we equip functor the category with the Day convolution product, we have a symmetric monoidal equivalence.
The same holds for the category of coproduct-preserving additive functors $(\mathsf{Perm}^G)^\op\to \Ab$,
\end{cor}

\begin{proof}
We know from Theorem~\ref{thm:burnside_bouc_R} that $k^G_*:\mathsf{perm}^G\simeq \mathcal B_{R^G}^G$ as tensor categories, so this is just Bouc's functorial picture for $R^G$-Mackey modules (\S\ref{subsec:functorial}).
Day convolution provides the correct tensor structure by construction, \emph{cf}.\ Remark~\ref{rem:alt_def_tensor}. 
\end{proof}

\begin{remark}
\label{rem:stable_htpy}
Corollary \ref{cor:ident_mackey_modules} should be compared with the following result, see~\cite{lewis-may-steinberger}*{Proposition V.9.6}: the Burnside category $\mathcal B=\mathcal B_{Bur}$ is equivalent to the full subcategory in the stable homotopy category of $G$-equivariant spectra,  $\mathsf{SH}^G$,
with objects all suspension spectra $\Sigma^\infty X_+$ for $X\in G\textrm-\set$. 
The authors of \emph{loc.\,cit.}\ define Mackey functors for a compact Lie group~$G$ precisely so that the analogous statement remains true in this case. 
It would be interesting to know whether the same definition proves useful for the study of $\KK^G$ when~$G$ is a  compact Lie group. 
\end{remark}

\begin{remark}
In principle, it must be possible to prove Theorem~\ref{thm:burnside_bouc_R} directly, without appealing to Theorem~\ref{thm:fully_faithful}. 
First notice that 
\begin{align*}
\label{iso_dir}
 \KK^G(C(G/H), C(G/L)) 
& \cong \KK^H(\unit , \Res{G}{H} C(G/L) ) \\
& \cong \bigoplus_{x\in [H\backslash G/L]} R(H\cap {}^x\! L) \\
& = R^G(G/H \times G/L) \\
& \smash{\stackrel{\textrm{def.}}{=} } \mathcal B_{R^G} (G/H, G/L) 
\end{align*}
 for all $H,L\leqslant G$, by the $(\Ind{}{},\Res{}{})$-adjunction in KK-theory.
 Then it remains ``only'' to prove that this identification takes the composition of~$\KK^G$ to the composition of~$\mathcal B_{R^G}$. 
But this seems like a lot of work: the Kasparov product is famously difficult to compute explicitly (although, admittedly, we are dealing here with an easy special case), and the explicit formula for the composition in the Burnside-Bouc category is also rather involved (see \cite{bouc:green_functors}*{\S3.2}).
In order to do this, one could perhaps use the correspondences of~\cite{emerson-meyer:corr} and their geometric picture of the Kasparov product.
Anyway, once Theorem~\ref{thm:burnside_bouc_R} is proved it is then possible to use abstract considerations to derive from it Theorem~\ref{thm:fully_faithful}, rather than the other way round. 
\end{remark}

\section{Relative homological algebra and $G$-cell algebras}
\label{sec:rel_hom}

We begin by recalling from~\cite{meyernest_hom} and~\cite{meyer:hom} a few definitions and results of \emph{relative homological algebra in triangulated categories}. 
This will allow us to establish some notation that will be used throughout the rest of the article.  

In the following, let $\mathcal T$ be a triangulated category endowed with arbitrary coproducts; for simplicity, we still assume that the shift 
functor~$\mathcal T\to \mathcal T$, $A\mapsto A[1]$, is a strict automorphism, rather than just a self-equivalence.

\begin{defi}
\label{hyp:rel}
It will be convenient to define the \emph{periodicity}, written~$\pi$, of the shift functor~$[1]\colon \mathcal T\to \mathcal T$ to be the smallest positive integer~$\pi$ such that there exists an isomorphism $[n]\cong \id_{\mathcal T}$, if such an integer exists; if it does not, we set~$\pi:=0$.
\end{defi}

\subsection{Recollections and notation}
A \emph{stable} abelian category is an abelian category equipped with an automorphism $M\mapsto M[1]$, called \emph{shift}.
A \emph{stable homological functor} is an additive functor $F\colon \mathcal T\to \mathcal A$ to a stable abelian category~$\mathcal A$, which commutes with the shift and which sends distinguished triangles to exact sequences in~$\mathcal A$.
In particular, a stable homological functor is homological in the usual sense. 

Conversely, if $F\colon \mathcal T\to \mathcal A$ is a homological functor to some abelian category~$\mathcal A$, then
we can construct a stable homological functor 
\[
F_* \colon \mathcal T\to \mathcal A_{\Z/\pi}
\]
as follows (recall that we allow $\pi=0$, in which case we have $\Z/\pi=\Z$). 
Here~$\mathcal A_{\Z/\pi}$ denotes the stable abelian category of $\Z/\pi$-graded objects in~$\mathcal A$ with degree preserving morphisms. As a category, it is simply the product $\mathcal A_{\Z/\pi}=\prod_{i\in\Z/\pi} \mathcal A$; 
we write $M_i$ for the $i$-th component of an object~$M$, and similarly for morphisms. 
The shift functor is given by $(M[1])_i := M_{i-1}$, and \emph{ditto} for morphisms.
Then we define~$F_*$ by $F(A)_i = F_i(A) := F(A[-i])$.

\begin{remark}
This choice of degree follows the usual (\emph{ho}mological) indexing conventions, according to which a distinguished triangle $A\to B\to C\to A[1]$ gives rise to a long exact sequence of the form $\ldots \to F_iA \to F_iB\to F_iC\to F_{i-1}A\to \ldots$. 
Note that, if instead $F\colon \mathcal T^{\op}\to \mathcal A$ is a contravariant homological functor, then the usual convention requires us to write indices \emph{up}, $F^i(A):=F(A[-i])$, to indicate that differentials now increase degree: 
$\ldots F^{i-1}(A)\to F^i(C)\to F^i(B)\to F^i(A)\ldots $. 
If one must insist in using homological notation (as we will do later with graded Yoneda and graded Ext groups), then one uses the conversion rule $F^i=F_{-i}$.
\end{remark}

A \emph{homological ideal}~$\mathcal I$ in~$\mathcal T$ is the collection of morphisms of~$\mathcal T$ vanishing under some stable homological functor~$H$: 
\[
\mathcal I= \ker H := \{f\in \Mor(\mathcal T)\mid F(f)=0 \} .
\]
Thus in particular $\mathcal I$ is a categorical ideal which is closed under shifts of maps.
Note that different stable homological functors~$H$ can define the same homological ideal~$\mathcal I$, but it is the latter datum that is of primary interest and will determine all ``relative'' homologico-algebraic notions.\footnote{There is an elegant axiomatic approach due to Beligiannis~\cite{beligiannis:relative} that does justice to this observation, but we will not use it here.} 
A homological functor $F\colon \mathcal T\to \mathcal A$ is \emph{$\mathcal I$-exact} if $\mathcal I\subseteq \ker F$.
An object $P\in \mathcal T$ is \emph{$\mathcal I$-projective} if $\mathcal T(P,\variable)\colon \mathcal T\to \Ab$ is $\mathcal I$-exact.
An \emph{$\mathcal I$-projective resolution} of an object $A\in \mathcal T$ is a diagram
$\ldots P_n\to P_{n-1}\to \ldots\to P_1\to P_0\to A\to 0$ in~$\mathcal T$ such that each $P_n$ is $\mathcal I$-projective and such that the sequence is \emph{$\mathcal I$-exact} in a suitable sense (see \cite{meyernest_hom}*{\S3.2}).

Let $F\colon \mathcal T\to \mathcal A$ be an additive (usually homological) functor to an abelian category, and let $n\geqslant 0$ be a nonnegative integer.
The \emph{$n$-th  $\mathcal I$-relative left derived functor of~$F$}, written $\mathsf L_n^\mathcal I F$, is the functor $\mathcal T\to \mathcal A$ obtained by taking an object $A\in \mathcal T$, choosing a projective resolution $P_\bullet$ for it, applying~$F$ to the complex~$P_\bullet$ and taking the $n$-th homology of the resulting complex in~$\mathcal A$ --- in the usual way.
In the case of a contravariant functor, $F\colon \mathcal T^\op\to \mathcal A$, we can still use $\mathcal I$-projective resolutions in the same way in~$\mathcal T$ to define the $\mathcal I$-relative \emph{right} derived functors $\mathsf R^n_\mathcal I F \colon \mathcal T^\op\to \mathcal A$.

\begin{remark}
One of course has to prove that the recipes for $\mathsf L_n^\mathcal IF$ and $\mathsf R^n_\mathcal I F$ yield well-defined functors. This is always the case -- as in our examples -- as soon as there are enough $\mathcal I$-projective objects, in the precise sense that for every $A\in \mathcal T$ there exists a morphism $P\to A$ fitting into a distinguished triangle $B\to P\to A\to B[1]$ where~$P$ is $\mathcal I$-projective and $(A\to B[1])\in \mathcal I$.
All our examples have enough $\mathcal I$-projectives but possibly not enough $\mathcal I$-injectives, which causes the above asymmetrical definition of derived functors.
\end{remark}

\begin{remark}
It is immediate from the definitions that one may stabilize either before or after taking derived functors, namely: 
$(\mathsf L_n^\mathcal I F)_*=\mathsf L_n^\mathcal I (F_*)$ and 
$(\mathsf R^n_\mathcal I F)_*=\mathsf R^n_\mathcal I (F_*)$.
\end{remark}

\subsection{The graded restricted Yoneda functor}
\label{subsec:general}

Assume now that we are given an (essentially) small set $\mathcal G\subseteq \mathcal T$ of \emph{compact} objects; that is, the functor $\mathcal T(X,\variable)\colon \mathcal T\to \Ab$ commutes with arbitrary coproducts for each $X\in \mathcal G$.  

Our goal is to understand the homological algebra in~$\mathcal T$ relative to~$\mathcal G$, that is, relative to the homological ideal
\[
\mathcal I 
:= \bigcap_{X \in \mathcal G} \ker \mathcal T(X, \variable)_*
= \{ f\in \Mor (\mathcal T) \mid \mathcal T(X[i],f)=0 \; \forall i\in \Z/\pi, X\in \mathcal G\}
.
\]
The reason we bother with this generality is that, already at this level, Ralf Mayer's ABC spectral sequence~\cite{meyer:hom} specializes to a pleasant-looking universal coefficient spectral sequence (see Theorem~\ref{thm:abstract_uct} below). 


Let $\mathcal T(A,B)_*=\{ \mathcal T(A[i], B) \}_{i\in \Z/\pi}$ denote the graded Hom in~$\mathcal T$ induced by the shift automophism, and let~$\mathcal T_*$ denote the $\Z/\pi$-graded category with the same objects as~$\mathcal T$ and composition given by
\begin{align*}
& \mathcal T(B[j], C) \times \mathcal T(A[i],B) \to \mathcal T(A[i+j], B)  \\
& \quad \quad (g , f) \quad \mapsto \quad gf :=  g \circ f[j] \quad .
\end{align*}
Similarly, denote by $\mathcal G_*$ the full graded subcategory of~$\mathcal T_*$ containing the objects of~$\mathcal G$.
Let $\GrMod \textrm-  \mathcal G_*$ be the category of \emph{graded right $\mathcal G_*$-modules}.  
Its objects are the degree-preserving functors $M\colon (\mathcal G_*)^{\op}\to (\Ab_{\Z/\pi})_*$ into the \emph{graded} category of graded abelian groups, and its morphisms are grading preserving natural transformations $\varphi: M\to M'$, \ie, families $\varphi_{i,X}\colon M_i(X)\to M_i'(X)$ ($i\in \Z/\pi, X\in \mathcal G$) of homomorphisms commuting with maps $M(f)$ of all degrees.
Note that $\GrMod \textrm-  \mathcal G_*$ is a stable abelian category with shift functor given by  $(M[k])_i:= M_{i-k}$ and $(f[k])_i:= f_{i-k}$.

Every $A\in \mathcal T$ defines a graded $\mathcal G_*$-module 
$h_*(A) := \{ \; \mathcal T((\variable )[i], A)\at_{\mathcal G_*} \; \}_{i\in Z/\pi}$
 in a natural way, so that we get a \emph{\textup(restricted\textup) Yoneda functor}
 \[
h_* \colon \mathcal T \longrightarrow \GrMod\textrm- \mathcal G_* 
\]
which is stable homological and moreover preserves coproducts, since the objects of~$\mathcal G$ are compact.

\begin{lemma}
\label{lemma:enriched_Yoneda}
There is a natural isomorphism of $\Z/\pi$-graded abelian groups 
\begin{equation}
\GrMod\textrm- \mathcal G_*( h_*(X), M)_* \cong M(X) 
\end{equation}
for all $X\in \mathcal G$ and all $M\in \GrMod\textrm-\mathcal G_*$, which sends the natural transformation $\varphi:h_*(X)[i]\to M$ to the element $\varphi_{i,X}(1_X)\in M_i(X)$.
\end{lemma}

\begin{proof}
This follows from the Yoneda lemma for $\Z/\pi$-graded $\Z$-linear categories, \ie, for categories enriched over the closed symmetric monoidal category~$\Ab_{\Z/\pi}$ (see~\cite{kelly:enriched_book}). It can also be easily proved by hand.
\end{proof}

We see in particular that the collection $\{h_*(X)[i]\mid i\in \Z/\pi , X\in \mathcal G\}$ forms a set of projective generators for~$\Gr\Mod\textrm- \mathcal G_*$.

We also note that~$h_*$ detects the vanishing of any object in  the localizing triangulated subcategory 
\[
\Cell(\mathcal T,\mathcal G):=  \langle \mathcal G \rangle_\loc \, \subseteq\, \mathcal T
\]
generated by~$\mathcal G$.

\begin{example}
\label{example:KKG}
Let $\mathcal T= \KK^G$ be the equivariant Kasparov category of a finite group~$G$, and let 
$\mathcal G:= \{C(G/H) \mid H\leqslant G \}$.
By Bott periodicity, we must grade over~$\Z/2$.
Then $\Cell(\mathcal T,\mathcal G)= \Cell^G$ , and the stable abelian category $\GrMod \textrm- \mathcal G_*$ is just $R^G\textrm-\Mack_{\Z/2}$, the category of Mackey modules over the representation Green functor for~$G$.
This is because in this case $\mathcal G(X[1],Y)=0$ for all $X,Y\in \mathcal G$, so that $\GrMod \textrm- \mathcal G_*$ is quite simply the category of $\Z/2$-graded objects in the usual ungraded module category $\Mod\textrm-\mathcal G$, and we know from Corollary~\ref{cor:ident_mackey_modules} that the latter is equivalent to $R^G\textrm- \Mack$.
If $G=1$ is the trivial group, then $\Cell(\mathcal T,\mathcal G)$ is the classical Bootstrap category of separable $\Cstar$-algebras~\cite{rs}, and $\Gr\Mod \textrm- \mathcal G_*$ is the category of graded abelian groups, $\Ab_{\Z/2}$.
\end{example}

\begin{prop}
\label{prop:univ_exact}
The restricted Yoneda functor 
$h_* \colon \mathcal T  \to \GrMod\textrm-\mathcal G_* $
is the universal $\mathcal I$-exact stable homological functor on~$\mathcal T$.
In particular, $h_*$ induces an equivalence between the category of $\mathcal I$-projective objects in~$\mathcal T$ and that of  projective graded right $\mathcal G$-modules:
\[
h_* \colon \Projobj(\mathcal T,\mathcal I) \simeq 
 \Projobj(\GrMod\textrm-\mathcal G_*) ,
\]
and for every $A\in \mathcal T$ it induces a bijection between $\mathcal I$-projective resolutions of~$A$ in~$\mathcal T$ and projective resolutions of the graded $\mathcal G$-module~$h_*(A)$.

\end{prop}

\begin{proof}
Since the stable homological functor~$h_*$ is $\mathcal I$-exact by definition, and
since idempotent morphisms in~$\mathcal T$ split (because of the existence of countable coproducts, \cite{bokstedt_neeman}*{Prop.~3.2}), 
we may use the criterion of \cite{meyernest_hom}*{Theorem~57}. 
Since the abelian category $\GrMod \textrm-\mathcal G_*$ has enough projectives, it remains to show the existence of a partial right-inverse and partial left adjoint to~$h_*$ defined on projective modules, \ie, the existence of a functor 
$\ell \colon \Projobj(\GrMod\textrm-\mathcal G_*) \to \mathcal T$ and two natural isomorphisms
\[
 h_*\circ \ell (P)\cong P
\quad \textrm{and} \quad
\mathcal T(\ell P,B) \cong \GrMod\textrm-\mathcal G_* (P,h_*B)
\quad (P\in \Projobj\textrm-\mathcal G_* \; , \; B\in \mathcal T).
\]
Since the projective objects in $\GrMod\textrm-\mathcal G_*$ are additively generated by the representables $ h_*(X[i]) = h_*(X)[i]$, for $X\in \mathcal G$ and $i\in \Z/\pi$, by \cite{meyernest_hom}*{Remark~58} we need only define the functor~$\ell$ and the natural isomorphisms on the full subcategory $\{h_*(X)[i] \mid X\in \mathcal G, i\in \Z\} \subset \GrMod\textrm-\mathcal G_*$.
For each such object, we set $\ell(h_*(X)[i]):= X[i]$, so that indeed 
 $h_*\ell (h_*(X)[i]) = h_*(X)[i]$, 
 and also
 \begin{align*}
 \mathcal T(\ell (h_*X[i]) , B)
  = h_0(B)(X[i])  
  \cong \GrMod\textrm- \mathcal G_* (h_*X[i] , h_*B)
\end{align*}
 by the Yoneda Lemma~\ref{lemma:enriched_Yoneda}.
 It is now evident how to extend~$\ell$ on morphisms.
\end{proof}

\begin{remark}
\label{rem:aleph}
If $\mathcal T$ does not have all set-indexed coproducts, the same argument still works if instead the following two hypotheses are satisfied:  
\begin{enumerate}
\item
There exists an uncountable regular cardinal~$\aleph$, such that~$\mathcal T$ has all small${}_{\aleph}$ coproducts, \ie, those indexed by sets of cardinality strictly smaller than~$\aleph$. (In particular $\mathcal T$ has at least all countable coproducts.)

\item
Every object $X\in \mathcal G$ in our generating set is \emph{small${}_{\aleph}$}, that is, the functor $\mathcal T(X,\variable)$ commutes with small${}_{\aleph}$ coproducts and sends every object $A\in \mathcal T$ to a  small${}_{\aleph}$ set:
$|\mathcal T(X,A)| < \aleph$.
\end{enumerate}
Then Proposition~\ref{prop:univ_exact} remains true, with precisely  the same proof, if in its statement we substitute $\GrMod\textrm-\mathcal G_*$ with the category of \emph{small${}_{\aleph}$} graded $\mathcal G_*$-modules, that is, those $M\in \GrMod\textrm-\mathcal G_*$ such that each $M(X)_i$ is small${}_{\aleph}$.
For our application to KK-theory, we will have to use this $\aleph$-relative version of the statement with~$\aleph=\aleph_1$.
\end{remark}

For the following next general statements, we may either assume that~$\mathcal T$ has arbitrary coproducts and the objects of~$\mathcal G$ are compact, or that $\mathcal T$ and~$\mathcal G$ satisfy the hypothesis of Remark~\ref{rem:aleph}.

\begin{notation}
Let 
$\Ext_{\mathcal G_*}^n(M,N)_*$
be the graded Ext functor in $\GrMod\textrm - \mathcal G_*$. 
In other words, $\Ext_{\mathcal G_*}^n(\variable,N)_*$ denotes the right derived functors of the \emph{graded} Hom functor $\GrMod\textrm -\mathcal G_*(\variable, N)_*\colon \GrMod\textrm-\mathcal G_* \to \Ab_{\Z/\pi}$; as usual, it is computed by projective resolutions of graded $\mathcal G$-modules. 
If, as in Example~\ref{example:KKG}, the category $\mathcal G_*=\mathcal G$ has only  maps in degree zero, then $\GrMod\textrm-\mathcal G_*= (\Mod\textrm-\mathcal G)_{\Z/\pi}$, and we may compute the graded Ext in terms of the ungraded Ext functors, according to the formula
\[
\Ext^n_{\mathcal G_*}(M,N)_\ell
=
\bigoplus_{i+j=\ell} \Ext^n_\mathcal G(M_i, M_j)
\qquad (n\in \N, \ell\in \Z/\pi ).
\]
\end{notation}

\begin{prop}
\label{prop:comput_rel_der}

If $F$ is a homological functor $F\colon \mathcal T^{\op}\to  \Ab$ sending coproducts in~$\mathcal T$ to products of abelian groups, 
then there are natural isomorphisms
\[
\mathsf R^n_{\mathcal I}F^* \cong \Ext^n_{\mathcal G_*}(h_*(\variable) , F\at_{\mathcal G_*})_{-*}
\qquad (n\in \N) 
\]
computing its right $\mathcal I$-relative derived functors.
Here $F\at_{\mathcal G_*}\colon \mathcal G_*\to \Ab_{\Z/2}$ denotes the graded $\mathcal G_*$-module obtained by considering the restriction of~$F$ to the full subcategory $\{ X[i] \mid X\in \mathcal G, i\in \Z/\pi \}\subseteq \mathcal T$.
\end{prop}

\begin{proof}
By Proposition~\ref{prop:univ_exact}, we know that $h_* \colon \mathcal T  \to \GrMod\textrm-\mathcal G_* $ is the universal $\mathcal I$-exact functor.
By \cite{meyernest_hom}*{Theorem~59}, there exists (up to canonical isomorphism) a unique left exact functor $\overline F\colon (\GrMod\textrm- \mathcal G_*)^{\op} \to \Ab$ such that $\overline F\circ h_*(P)= F(P) $ for every $\mathcal I$-projective object~$P$ of~$\mathcal T$.
Since $h_*$ induces a bijection between $\mathcal I$-projective resolutions of~$A\in \mathcal T$ and projective resolutions of $h_*(A)\in \GrMod\textrm-\mathcal G_*$, there follows easily the existence of isomorphisms
\begin{equation}
\label{ident_contrav_rel_der}
\mathsf R^n_{\mathcal I}F \cong \mathsf R^n \overline F \circ h_* 
\qquad (n\in \N).
\end{equation}
By Lemma~\ref{lemma:enriched_Yoneda} for the module $M=F\at_{\mathcal G_*}$, there are natural isomorphisms
\begin{align*}
F(X[i]) 
& = F^{-i}(X) \\
& = (F\at_{\mathcal G_*})_i(X) \\
& \cong \GrMod\textrm-\mathcal G_*( h_*X, F\at_{\mathcal G_*})_i \\
& = \GrMod\textrm-\mathcal G_*( h_*(X[i]), F\at_{\mathcal G_*}) 
\end{align*}
for all $X\in \mathcal G$ and $i\in \Z/\pi$. 
Since every $\mathcal I$-projective object is a direct summand of a coproduct of such~$X[i]$,
we may extend this additively to a natural isomorphism 
\begin{align}
\label{yoneda_identif}
F(P) 
 \cong \GrMod\textrm-\mathcal G_*( h_*(P), F\at_{\mathcal G_*}) 
\end{align}
for all $P\in \Projobj(\mathcal T,\mathcal I)$.
Moreover, the Hom functor $\GrMod\textrm-\mathcal G_*(\variable , F\at_{\mathcal G_*})$ is left exact. 
Hence by \eqref{yoneda_identif} we can identify $\GrMod\textrm-\mathcal G_*(\variable , F\at_{\mathcal G_*})$ with~$\overline F$, because of the uniqueness property of the latter.
By injecting this knowledge into~\eqref{ident_contrav_rel_der} we get the required isomorphisms.
\end{proof}

As an important special case, we can now compute the $\mathcal I$-relative Ext functors (\emph{cf.}~\cite{meyer:hom}*{p.\,195}).

\begin{cor}
\label{cor:rel_ext}
There are isomorphisms
\[
\mathsf R^n_{\mathcal I} ( \mathcal T(\variable, B) ) \cong \Ext^n_{\mathcal G_*} (h_*(\variable) , h_*B)_0
\qquad (n\in \N)
\]
of functors $\mathcal T^\op\to \Ab$ for all $B\in \mathcal T$.
\end{cor}

\begin{proof}
For every $B\in \mathcal T$, the functor $F:= \mathcal T(\variable, B)\colon \mathcal T^\op\to \Ab$ satisfies the hypothesis of Proposition~\ref{prop:comput_rel_der}, and in this case we have $F\at_{\mathcal G_*}= h_*(B)$ by definition.
Now we apply the proposition and look at degree zero.
\end{proof}

\subsection{The universal coefficients spectral sequence}

We are ready to prove the following general form of the universal coefficient theorem. 
We do not claim any originality for this result, as it is already essentially in \cite{christensen:ideals} and~\cite{meyer:hom}.

\begin{thm}
\label{thm:abstract_uct}
Let $\mathcal T$ be a triangulated category with arbitrary coproducts and let~$\mathcal G$ be a small set of its compact objects ---  or assume the $\aleph$-relative versions of this hypothesis, as in Remark~\ref{rem:aleph}.
For all objects $A,B\in \mathcal T$, there is a cohomologically indexed right half-plane spectral sequence of the form
\[
E_2^{p,q} = \Ext^{p}_{\mathcal G_*}(h_*A,h_*B)_{-q}
\quad \stackrel{n=p+q}{\Longrightarrow} \quad \mathcal T(A[n],B)
\]
depending functorially on~$A$ and~$B$.

The spectral sequence converges conditionally  \textup(\cite{boardman}\textup) if $A\in \Cell(\mathcal T,\mathcal G)\smash{\stackrel{\textrm{def.}}{=}} \langle \mathcal G\rangle_{\loc}\subseteq \mathcal T$.
We have strong convergence if~$A$ is moreover $\mathcal I^\infty$-projective, that is, if $\mathcal T(A,f)=0$ for every morphism $f$ which can be written, for every~$n\geqslant1$, as a composition of~$n$ morphisms each of which vanishes under~$h_*$. 

If $A$ belongs to $\Cell(\mathcal T,\mathcal G)$ and moreover has an $\mathcal I$-projective resolution of finite length~$m\geqslant1$ (equivalently: if $A\in \Cell(\mathcal T,\mathcal G)$ and~$h_*A$ has a projective resolution in $\Gr\Mod\textrm-\mathcal G_*$ of length~$m$), then the spectral sequence is confined in the region $0\leqslant p \leqslant m$ and therefore collapses at the $(m+1)$-st page $E^{*,*}_{m+1}=E^{*,*}_\infty$.
\end{thm}

\begin{proof}
We define our spectral sequence to be the ABC spectral sequence of~\cite{meyer:hom} associated to the triangulated category $\mathcal T$, its homological ideal $\mathcal I= \ker h_*$, the contravariant homological functor $F=\mathcal T(\variable,B)$ into abelian groups, and the object~$A\in \mathcal T$; the hypotheses that $\mathcal T$ has countable coproducts, that $\mathcal I$ is closed under them, and that~$F$ sends them to products, are all satisfied. 
By \cite{meyer:hom}*{Theorem~4.3}, the ABC spectral sequence is (from the second page onwards) functorial in~$A$, and ours is clearly also functorial in~$B$ by construction. 
Moreover, its second page contains the groups $E^{p,q}_2=\mathsf R_\mathcal I^pF^q(A)$, which take the required form by Corollary~\ref{cor:rel_ext}.

The criterion for strong convergence is proved in \cite{meyer:hom}*{Proposition~5.2} (where, in the notation of \emph{loc.\,cit.}, $A=LA$ and $\mathbb R F = F$ because $A\in \Cell(\mathcal T,\mathcal G)$), and the criterion for collapse is part of \cite{meyer:hom}*{Proposition~4.5}.   

Conditional convergence is proved as in \cite{christensen:ideals}*{Proposition~4.4}. The hypothesis of \emph{loc.\,cit.\ }is that $\mathcal G$-projective objects generate, \emph{i.e.}, that $\Cell(\mathcal T,\mathcal G)=\mathcal T$. However, for fixed $A$ and~$B$, the argument only uses that $A\in \Cell(\mathcal T,\mathcal G)$, not~$B$: this still implies that $X_k\in \Cell(\mathcal T,\mathcal G)$ for all the stages of the Adams resolution \cite{christensen:ideals}*{(4.1)}, \emph{i.e.}, of the phantom tower \cite{meyer:hom}*{(3.1)}, and the conclusion follows exactly with the same proof.
\end{proof}

Specializing Theorem~\ref{thm:abstract_uct} to Example~\ref{example:KKG}, we obtain the first of the results promised in the Introduction.

\begin{thm}
\label{thm:UCT}
Let $G$ be a finite group. 
For every $A,B\in \KK^G$ such that $A$ is $G$-cell algebra, and depending functorially on them, there exists a cohomologically indexed, right half plane, conditionally convergent spectral sequence 
\[
E^{p,q}_2 = \Ext^p_{R^G}( k^G_*A , k^G_*B )_{-q} 
\quad
\stackrel{n= p + q}{\Longrightarrow} 
\quad
\KK^G_n(A, B)
\]
which converges strongly or collapses under the same hypothesis of Theorem~\ref{thm:abstract_uct}.
\end{thm}

\begin{proof}
Since $\KK^G$ only has countable coproducts, we adopt the hypotheses of Remark~\ref{rem:aleph} with~$\aleph:=\aleph_1$. Note that the generators $\mathcal G=\{C(G/H)\mid H\leqslant G\}$ are compact${}_{\aleph_1}$ by Proposition~\ref{prop:cptly_gen}.
The universal $\mathcal G$-exact stable homological functor of Proposition~\ref{prop:univ_exact} is just our $\smash{ k^G_*: \KK^G \to R^G\textrm-\Mack_{\Z/2}^{\aleph_1} }$ (where the ``$\aleph_1$'' indicates that we must restrict attention to countable modules),
and the rest follows.
\end{proof}

\subsection{The K\"unneth spectral sequence}
We have a fairly good idea of what should be the most appropriate level of abstraction for proving a nice general K\"unneth spectral sequence, similar to the general universal coefficient spectral sequence of~\S\ref{subsec:general}. But this would involve inflicting on the reader more abstract nonsense than might be decently included in this article, and we therefore reserve such thoughts for a different place and a future time.

For Kasparov theory, in any case, we have the following.

\begin{thm}
\label{thm:KT}
Let $G$ be a finite group.
For all separable $G$-$\Cstar$-algebras~$A$ and~$B$, there is a homologically indexed right half-plane spectral sequence of the form
\[
E^2_{p,q} = \Tor_{p}^{\mathcal G_*}(k^G_*A,k^G_*B)_{q}
\quad \stackrel{n=p+q}{\Longrightarrow} \quad K^G_n(A\otimes B)
\]
which is strongly convergent whenever $A\in \Cell^G$.  
\end{thm}

The key is to show that the generators in~$\mathcal G$ are sufficiently nice with respect to the universal $\mathcal I$-exact functor, so that we may correctly identify the second page.  

\begin{lemma}
\label{lemma:iso_proj}
For every $H\leqslant G$, there is an isomorphism of graded $R^G$-modules
\[
\varphi_H\colon k^G_*(C(G/H)\otimes B) \cong  k^G_*(C(G/H)) \boxpr_{R^G} k^G_*(B)
\]
natural in $B\in \Cstaralg^G$.
\end{lemma}

\begin{proof}
We define $\varphi_H$ by the following commutative diagram.
\begin{equation*}
\xymatrix{
k^G_*(C(G/H)\otimes B) \ar@{=}[d] \ar@{..>}[r]^-{\varphi_{H}} &
 k^G_*(C(G/H)) \boxpr_{\!R^G}\, k^G_*(B) \ar[d]_{\cong}^{\textrm{Lemma~\ref{lemma:Ind_k}} } \\
k^G_*(\Ind{H}{G}(\unit^H)\otimes B) \ar[d]_{\textrm{Frobenius \eqref{frob:Cstar}}}^{\cong} &
 \Ind{G}{H}(R^H) \boxpr_{\!R^G}\, k^G_*(B) \ar[d]_{\cong}^{\textrm{Frobenius  \ref{prop:higher_frob} }} \\
k^G_* (\Ind{H}{G} ( \unit^H \otimes \Res{G}{H}(B))) \ar[d]^{\cong} &
 \Ind{H}{G}(R^H \boxpr_{\!R^H}\, \Res{G}{H}(k^G_*(B))) \ar[d]_{\cong} \\
k^G_*(\Ind{H}{G}\Res{G}{H}((B))) \ar[r]_-{\cong}^-{\textrm{Lemma \ref{lemma:Ind_k}}}&
 \Ind{H}{G}\Res{G}{H}(k^G_*(B))
}
\end{equation*}
Because each isomorphism is natural in~$B$, their composition is too.
\end{proof}

\begin{prop}
\label{prop:comput_rel_tor}
For the stable homological functor 
$
F_*:= k^G_*( \variable \otimes B) 
\colon \KK^G \to \Ab_{\Z/2}
$, 
there are canonical isomorphisms
\[
\mathsf L^{\mathcal I}_nF_* \cong \Tor_n^{R^G}(k^G_*(\variable), k^G_*(B))_*  \qquad (A \in \KK^G , n\in \Z )
\]
of functors $\KK^G\to \Ab_{\Z/2}$
between its $\mathcal I$-relative left derived functors and the left derived functors of the tensor product $\boxpr_{R^G}$ of graded $R^G$-Mackey modules. 
Moreover, if we equip both sides with the naturally induced $R^G$-action, these isomorphisms become isomorphisms of functors $\KK^G \to R^G\textrm- \Mack_{\Z/2}$.
\end{prop}

\begin{proof}
The proof is quite similar to that of Proposition~\ref{prop:comput_rel_der}, but let us go again through the motions.
By \cite{meyernest_hom}*{Theorem~59}, there exists (up to canonical isomorphism) a unique right exact functor 
$\overline{F_*}\colon R^G\textrm- \Mack_{\Z/2}\to \Ab$ such that  $\overline{F_*}\circ k^G_*(P)= F_*(P)$ for all $\mathcal I$-projective objects $P\in \KK^G$. Moreover, it follows that there are isomorphisms 
\begin{align}
\label{two_der_functors}
\mathsf L_n^\mathcal I F_* \cong \mathsf L_n\overline{F_*} \circ k^G_* 
\qquad (n\geqslant 0) .
\end{align}
Because of the isomorphisms $\varphi_H$ of Lemma~\ref{lemma:iso_proj}, and the consequent isomorphisms
\[
\varphi_H [1] \; \colon \;
k^G_*(C(G/H)[1]\otimes B) \cong k^G_*(C(G/H))[1]  \boxpr_{ R^G} k^G_*(B),
\]
we see that $k^G_*(\variable) \boxpr_{R^G}  k^G_*(B)$ agrees  with~$F_*$ on $\mathcal I$-projective objects. 
Therefore, since $(\variable ) \boxpr_{ R^G} k^G_*(B)$ is right exact, we can further identify 
 $\overline{F_*}= (\variable) \boxpr_{R^G} k^G_*(B)$.
Hence by taking left derived functors we deduce that the identification~\eqref{two_der_functors} takes the form of the claimed isomorphisms.
The last claim of the proposition is clear from the naturality of the isomorphisms. 
\end{proof}

\begin{proof}[Proof of Theorem~\ref{thm:KT}]
We define the K\"unneth spectral sequence to be the ABC spectral sequence of~\cite{meyer:hom} associated with the triangulated category $\mathcal T= \KK^G$, the homological ideal~$\smash{\mathcal I=\bigcap_{H\leqslant G} \ker(K^H_*\circ \Res{G}{H}) }$, the object~$A$, and the covariant homological functor $F= k^G( \variable \otimes B): \KK^G\to \Ab$; the hypotheses that $\mathcal T$ has countable coproducts, that $\mathcal I$ is closed under them and that $F$ preserves them are all satisfied.
The description of the second page follows from \cite{meyer:hom}*{Theorem~4.3} and the computation in Proposition~\ref{prop:comput_rel_tor}. The strong convergence to $F_n(A)=K^G_n(A\otimes B)$ follows from \cite{meyer:hom}*{Theorem~5.1} together with the hypothesis that~$A$ belongs to $\Cell^G$, namely, to the localizing subcategory generated by the $\mathcal I$-projective objects, which implies that $F(A[n])=\mathbb LF(A[n])$ in the notation of \emph{loc.\,cit}.
(Note that, contrary to the case of a contravariant homological functor, we do not need extra conditions on~$A$ for strong convergence.)
\end{proof}

\begin{bibdiv}
\begin{biblist}

\bib{beligiannis:relative}{article}{
   author={Beligiannis, Apostolos},
   title={Relative homological algebra and purity in triangulated
   categories},
   journal={J. Algebra},
   volume={227},
   date={2000},
   number={1},
   pages={268--361},
   issn={0021-8693},
}

\bib{bensonI}{book}{
   author={Benson, D. J.},
   title={Representations and cohomology. I},
   series={Cambridge Studies in Advanced Mathematics},
   volume={30},
   edition={2},
   note={Basic representation theory of finite groups and associative
   algebras},
   publisher={Cambridge University Press},
   place={Cambridge},
   date={1998},
   pages={xii+246},
   isbn={0-521-63653-1},
}


\bib{boardman}{article}{
   author={Boardman, J. Michael},
   title={Conditionally convergent spectral sequences},
   conference={
      title={Homotopy invariant algebraic structures},
      address={Baltimore, MD},
      date={1998},
   },
   book={
      series={Contemp. Math.},
      volume={239},
      publisher={Amer. Math. Soc.},
      place={Providence, RI},
   },
   date={1999},
   pages={49--84},
}

\bib{bokstedt_neeman}{article}{
   author={B{\"o}kstedt, Marcel},
   author={Neeman, Amnon},
   title={Homotopy limits in triangulated categories},
   journal={Compositio Math.},
   volume={86},
   date={1993},
   number={2},
   pages={209--234},
   issn={0010-437X},
}

\bib{bouc:green_functors}{book}{
   author={Bouc, Serge},
   title={Green functors and $G$-sets},
   series={Lecture Notes in Mathematics},
   volume={1671},
   publisher={Springer-Verlag},
   place={Berlin},
   date={1997},
   pages={viii+342},
   isbn={3-540-63550-5},
}

\bib{christensen:ideals}{article}{
   author={Christensen, J. Daniel},
   title={Ideals in triangulated categories: phantoms, ghosts and skeleta},
   journal={Adv. Math.},
   volume={136},
   date={1998},
   number={2},
   pages={284--339},
   issn={0001-8708},
}

\bib{day}{article}{
   author={Day, Brian},
   title={On closed categories of functors},
   conference={
      title={Reports of the Midwest Category Seminar, IV},
   },
   book={
      series={Lecture Notes in Mathematics, Vol. 137},
      publisher={Springer},
      place={Berlin},
   },
   date={1970},
   pages={1--38},
}


\bib{kkGarticle}{article}{
   author={Dell'Ambrogio, Ivo},
   title={Tensor triangular geometry and KK-theory},
   journal={J. Homotopy Relat. Struct.},
   volume={5},
   date={2010},
   number={1},
   pages={319-358},
   issn={},
}

\bib{dress:contributions}{article}{
   author={Dress, Andreas W. M.},
   title={Contributions to the theory of induced representations},
   conference={
      title={Algebraic $K$-theory, II: ``Classical'' algebraic $K$-theory
      and connections with arithmetic (Proc. Conf., Battelle Memorial Inst.,
      Seattle, Wash., 1972)},
   },
   book={
      publisher={Springer},
      place={Berlin},
   },
   date={1973},
   pages={183--240. Lecture Notes in Math., Vol. 342},
}



\bib{nccw}{article}{
   author={Eilers, S{\o}ren},
   author={Loring, Terry A.},
   author={Pedersen, Gert K.},
   title={Stability of anticommutation relations: an application of
   noncommutative CW complexes},
   journal={J. Reine Angew. Math.},
   volume={499},
   date={1998},
   pages={101--143},
   issn={0075-4102},
}

\bib{emerson-meyer:corr}{article}{
   author={Emerson, Heath},
   author={Meyer, Ralf},
   title={Bivariant K-theory via correspondences},
   journal={Adv. Math.},
   volume={225},
   date={2010},
   number={5},
   pages={2883--2919},
   issn={0001-8708},
}

\bib{kasparov_first}{article}{
   author={Kasparov, G. G.},
   title={The operator $K$-functor and extensions of $C^{\ast} $-algebras},
   language={Russian},
   journal={Izv. Akad. Nauk SSSR Ser. Mat.},
   volume={44},
   date={1980},
   number={3},
   pages={571--636, 719},
}

\bib{keller-vossieck:sous}{article}{
   author={Keller, Bernhard},
   author={Vossieck, Dieter},
   title={Sous les cat\'egories d\'eriv\'ees},
   language={French, with English summary},
   journal={C. R. Acad. Sci. Paris S\'er. I Math.},
   volume={305},
   date={1987},
   number={6},
   pages={225--228},
}

\bib{kelly:enriched_book}{article}{
   author={Kelly, G. M.},
   title={Basic concepts of enriched category theory},
   note={Reprint of the 1982 original [Cambridge Univ. Press, Cambridge;
   MR0651714]},
   journal={Repr. Theory Appl. Categ.},
   number={10},
   date={2005},
   pages={vi+137},
}


\bib{koehler:uct}{article}{
   author={K\"ohler, Manuel},
   title={Universal coefficient theorems in equivariant KK-theory},
   conference={
      title={PhD thesis, Georg-August-Universit\"at G\"ottingen},
   },
   book={
      series={},
      volume={},
      publisher={},
      place={},
   },
   eprint={http://webdoc.sub.gwdg.de/diss/2011/koehler/},
   date={2010},
   pages={},
}


\bib{lewis:mimeo}{article}{
   author={Lewis, L. Gaunce, Jr.},
   title={The theory of Green functors},
   journal={unpublished notes},
   volume={},
   date={1981},
   eprint={http://people.virginia.edu/~mah7cd/Foundations/main.html},
   pages={},
  }

\bib{lewis-mandell:uct}{article}{
   author={Lewis, L. Gaunce, Jr.},
   author={Mandell, Michael A.},
   title={Equivariant universal coefficient and K\"unneth spectral
   sequences},
   journal={Proc. London Math. Soc. (3)},
   volume={92},
   date={2006},
   number={2},
   pages={505--544},
}

\bib{lewis-may-steinberger}{book}{
   author={Lewis, L. G., Jr.},
   author={May, J. P.},
   author={Steinberger, M.},
   author={McClure, J. E.},
   title={Equivariant stable homotopy theory},
   series={Lecture Notes in Mathematics},
   volume={1213},
   note={With contributions by J. E. McClure},
   publisher={Springer-Verlag},
   place={Berlin},
   date={1986},
   pages={x+538},
}


\bib{lindner:remark}{article}{
   author={Lindner, Harald},
   title={A remark on Mackey-functors},
   journal={Manuscripta Math.},
   volume={18},
   date={1976},
   number={3},
   pages={273--278},
}

\bib{meyer_homom}{article}{
   author={Meyer, Ralf},
   title={Equivariant Kasparov theory and generalized homomorphisms},
   journal={$K$-Theory},
   volume={21},
   date={2000},
   number={3},
   pages={201--228},
}

\bib{meyer_cat}{article}{
   author={Meyer, Ralf},
   title={Categorical aspects of bivariant $K$-theory},
   conference={
      title={$K$-theory and noncommutative geometry},
   },
   book={
      series={EMS Ser. Congr. Rep.},
      publisher={Eur. Math. Soc., Z\"urich},
   },
   date={2008},
   pages={1--39},
}

\bib{meyer:hom}{article}{
   author={Meyer, Ralf},
   title={Homological algebra in bivariant $K$-theory and other triangulated
   categories. II},
   journal={Tbil. Math. J.},
   volume={1},
   date={2008},
   pages={165--210},
}

\bib{meyer_nest_bc}{article}{
   author={Meyer, Ralf},
   author={Nest, Ryszard},
   title={The Baum-Connes conjecture via localisation of categories},
   journal={Topology},
   volume={45},
   date={2006},
   number={2},
   pages={209--259},
}

\bib{meyernest_hom}{article}{
   author={Meyer, Ralf},
   author={Nest, Ryszard},
   title={Homological algebra in bivariant $K$-theory and other triangulated
   categories. I},
   conference={
      title={Triangulated categories},
   },
   book={
      series={London Math. Soc. Lecture Note Ser.},
      volume={375},
      publisher={Cambridge Univ. Press},
      place={Cambridge},
   },
   date={2010},
   pages={236--289},
}

\bib{phillips:free}{book}{
   author={Phillips, N. Christopher},
   title={Equivariant $K$-theory and freeness of group actions on $C^*$-algebras},
   series={Lecture Notes in Mathematics},
   volume={1274},
   publisher={Springer-Verlag},
   place={Berlin},
   date={1987},
   pages={viii+371},
}

\bib{rs}{article}{
   author={Rosenberg, Jonathan},
   author={Schochet, Claude},
   title={The K\"unneth theorem and the universal coefficient theorem for
   Kasparov's generalized $K$-functor},
   journal={Duke Math. J.},
   volume={55},
   date={1987},
   number={2},
   pages={431--474},
}

\bib{rosenberg-schochet:equiv}{article}{
   author={Rosenberg, Jonathan},
   author={Schochet, Claude},
   title={The K\"unneth theorem and the universal coefficient theorem for
   equivariant $K$-theory and $KK$-theory},
   journal={Mem. Amer. Math. Soc.},
   volume={62},
   date={1986},
   number={348},
   pages={vi+95},
}

\bib{thevenaz-webb:structure}{article}{
   author={Th{\'e}venaz, Jacques},
   author={Webb, Peter},
   title={The structure of Mackey functors},
   journal={Trans. Amer. Math. Soc.},
   volume={347},
   date={1995},
   number={6},
   pages={1865--1961},
   issn={0002-9947},
}

\bib{ventura}{article}{
   author={Ventura, Joana},
   title={Homological algebra for the representation Green functor for
   abelian groups},
   journal={Trans. Amer. Math. Soc.},
   volume={357},
   date={2005},
   number={6},
   pages={2253--2289 (electronic)},
   issn={0002-9947},
}


\bib{webb:guide}{article}{
   author={Webb, Peter},
   title={A guide to Mackey functors},
   conference={
      title={Handbook of algebra, Vol. 2},
   },
   book={
      publisher={North-Holland},
      place={Amsterdam},
   },
   date={2000},
   pages={805--836},
}

\end{biblist}
\end{bibdiv}
\end{document}